\documentclass[a4paper,11pt]{article}

\usepackage{geometry}
\usepackage[english]{babel}
\usepackage{amsmath}
\usepackage[utf8]{inputenc}
\usepackage[T1]{fontenc}
 
\usepackage{amsthm}
\usepackage{amssymb}
\usepackage[all]{xy}
\usepackage[,colorlinks=true,citecolor= DarkBlue,linkcolor=Black]{hyperref}
\usepackage[svgnames,table]{xcolor}
\usepackage{tikz,tkz-tab}
\usepackage{enumitem}

\theoremstyle{definition}
\newtheorem{definition}{Definition}[section]

\theoremstyle{plain}
\newtheorem{theorem}{Theorem}
\newtheorem*{theorem*}{Theorem}
\newtheorem{proposition}[definition]{Proposition}
\newtheorem*{proposition*}{Proposition}
\newtheorem{lemma}[definition]{Lemma}
\newtheorem*{lemma*}{Lemma}
\newtheorem{corollary}[definition]{Corollary}
\newtheorem*{conjecture}{Conjecture}

\newtheorem*{fact*}{Fact}

\theoremstyle{remark}
\newtheorem{remark}[definition]{Remark}

\newcommand{\R}{\mathbf{R}}
\newcommand{\Z}{\mathbf{Z}}

\newcommand{\gl}{\mathfrak{gl}}
\renewcommand{\sl}{\mathfrak{sl}}
\newcommand{\su}{\mathfrak{su}}

\newcommand{\aff}{\mathfrak{aff}}

\renewcommand{\a}{\mathfrak{a}}
\newcommand{\g}{\mathfrak{g}}
\newcommand{\h}{\mathfrak{h}}    
\newcommand{\m}{\mathfrak{m}}

\newcommand{\so}{\mathfrak{so}}
\newcommand{\p}{\mathfrak{p}}
\newcommand{\s}{\mathfrak{s}}
\renewcommand{\u}{\mathfrak{u}}

\newcommand{\e}{\mathrm{e}}
\renewcommand{\d}{\mathrm{d}}

\DeclareMathOperator{\Isom}{Isom}
\DeclareMathOperator{\Ad}{Ad}  
\DeclareMathOperator{\ad}{ad}

\DeclareMathOperator{\Ker}{Ker}

\DeclareMathOperator{\Span}{Span}

\DeclareMathOperator{\Kill}{Kill}
\DeclareMathOperator{\Int}{Int}

\DeclareMathOperator{\GL}{GL}
\DeclareMathOperator{\SL}{SL}
\DeclareMathOperator{\CO}{CO}

\DeclareMathOperator{\PSL}{PSL}
\DeclareMathOperator{\PO}{PO}

\DeclareMathOperator{\SO}{SO}
\DeclareMathOperator{\SU}{SU}

\DeclareMathOperator{\Aff}{Aff}

\DeclareMathOperator{\Conf}{Conf}

\DeclareMathOperator{\id}{id}
\DeclareMathOperator{\Supp}{Supp}

\renewcommand{\S}{\mathbf{S}}
\newcommand{\E}{\mathbf{E}}
\renewcommand{\H}{\mathbf{H}}
\newcommand{\dS}{\mathbf{dS}} 
\newcommand{\Ein}{\mathbf{Ein}} 
\newcommand{\X}{\mathbf{X}}

\newcommand{\N}{\mathcal{N}}

\renewcommand{\epsilon}{\varepsilon}
\renewcommand{\geq}{\geqslant}
\renewcommand{\leq}{\leqslant}
\renewcommand{\hat}{\widehat}  
\newcommand{\hx}{\hat{x}}
\renewcommand{\tilde}{\widetilde}
\renewcommand{\bar}{\overline}

\title{Lorentzian manifolds with a conformal action of $\SL(2,\R)$}
\author{Vincent Pecastaing}
\date{}

\begin{document}

\maketitle

\vspace*{-1cm}

\begin{center}
\today
\end{center}

\begin{center}
{\footnotesize Laboratoire de mathématique d'Orsay, UMR 8628 Université Paris-Sud et CNRS, Université Paris-Saclay, 91405 ORSAY Cedex, FRANCE. \textit{Email}: vincent.pecastaing@normalesup.org }
\end{center}

\vspace*{-.7cm}

\begin{center}
{\footnotesize \textit{Mathematics Subject Classification 2010}: 53A30, 53B30, 57S20, 37D40, 37D25.}
\end{center}

\begin{abstract}
We consider conformal actions of simple Lie groups on compact Lorentzian manifolds. Mainly motivated by the Lorentzian version of a conjecture of Lichnerowicz, we establish the alternative: Either the group acts isometrically for some metric in the conformal class, or the manifold is conformally flat - that is, everywhere locally conformally diffeomorphic to Minkowski space-time. When the group is non-compact and not locally isomorphic to $\SO(1,n)$, $n \geq 2$, we derive global conclusions, extending a theorem of \cite{frances_zeghib} to some simple Lie groups of real-rank $1$. This result is also a first step towards a classification of conformal groups of compact Lorentzian manifolds, analogous to a classification of their isometry groups due to Adams, Stuck and, independently, Zeghib \cite{adams_stuck,adams_stuckII,zeghib98}.
\end{abstract}

\tableofcontents

\section{Introduction}
\label{s:intro}
Given a geometric structure on a differentiable manifold $M$, an interesting problem consists in relating the algebraic and dynamical properties of its automorphism group to the geometry of the manifold. The question we are considering in this article is to infer geometric information from the dynamics of a simple Lie group, which is acting by preserving the conformal geometry defined by a Lorentzian metric $g$ on $M$. 

We already had investigated this problem in the \textit{real-analytic} case in \cite{article3}. The analyticity assumption enabled us to develop strong arguments based on the general behavior of local automorphisms of analytic rigid geometric structures, first described by Gromov \cite{gromov}, and then revisited by Melnick \cite{melnick} for Cartan geometries, see also \cite{article1}. However, these methods were not transposable to smooth structures, the conclusions of Gromov's theory being weaker for $\mathcal{C}^{\infty}$ structures. 

More generally, considering real-analytic rigid geometric structures reduces significantly the difficulty, be it at a local or global scale, and the corresponding smooth problem can be much more complicated to handle. For instance, a celebrated theorem of D'Ambra \cite{dambra88} on analytic, compact, simply connected Lorentzian manifolds, based on properties of local extensions of local Killing fields, is still open in the $\mathcal{C}^{\infty}$ case.

\vspace*{.2cm}

The main contribution of the present article is to introduce what we think to be a new approach in the study of conformal Lorentzian dynamics, valid in smooth regularity. We no longer use Gromov's theory, and the corresponding difficulty of the problem is now treated via the theory of non-uniformly hyperbolic dynamics.

\paragraph*{Ferrand-Obata theorem.} One of the main motivations for the study of conformal dynamics of Lie groups in Lorentzian signature comes from the Riemannian setting. A strong theorem due to Ferrand \cite{ferrand71,ferrand96} and Obata \cite{obata70} asserts that if a Lie group acts conformally and non-properly on a Riemannian manifold, then this manifold is conformally diffeomorphic to the round sphere $\S^n$ or the Euclidean space $\E^n$ of same dimension. Thus, the sphere being the conformal compactification of the Euclidean space, there is essentially one Riemannian manifold admitting a non-proper conformal action, and of course, this action is the one of a subgroup of the Möbius group on $\S^n$ or $\S^n \setminus \{p\}$, with $p \in \S^n$.

\vspace*{.2cm}

This theorem nicely illustrates the \textit{rigidity} of conformal dynamics and suggests that analogous phenomenon could be observed on other kinds of rigid geometric structures, especially conformal structures in other signatures. Non-properness of the action is no longer adapted in this context and a pertinent dynamical hypothesis is \textit{essentiality}.

\vspace*{.2cm}

Recall that two pseudo-Riemannian metrics $g$ and $g'$ on a manifold $M$ are said to be conformal if there exists a smooth function $\varphi	: M \rightarrow \R_{>0}$ such that $g' = \varphi g$. The conformal class of $g$ is $[g] = \{g', \ g' \text{ conformal to } g\}$, and a local diffeomorphism is said to be conformal if its differential preserves $[g]$. When $\dim M \geq 3$, the group of conformal diffeomorphisms of $(M,g)$ is a Lie transformation group, noted $\Conf(M,g)$.

\begin{definition}
Let $H < \Conf(M,g)$ be a Lie subgroup. We  say that \textit{$H$ acts inessentially on $M$}, or simply $H$ \textit{is inessential}, if there exists $g'$ conformal to $g$ such that $H$ acts on $M$ by isometries of $g'$. If not, we say that $H$ acts \textit{essentially}, or simply that $H$ \textit{is essential}.
\end{definition}

In fact, a Riemannian conformal action is essential if and only if it is non-proper (\cite{ferrand96}, Theorem A2), and Ferrand-Obata result concerns essential Riemannian groups. The question that naturally arises is whether or not there exists a unique pseudo-Riemannian manifold with an essential conformal group, or at least if one can classifies such manifolds. 

It turned out that the existence of an essential group is far less restrictive for non-Riemannian manifolds, even when the metric is Lorentzian: In \cite{alekseevsky}, Alekseevsky built many examples of Lorentzian metrics on $\R^n$ admitting an essential flow. In \cite{frances_ferrand_obata_lorentz}, Frances provided infinitely many examples of \textit{compact} Lorentzian manifolds whose conformal group is essential. See \cite{kuhnel_rademacher1,kuhnel_rademacher2,frances_lichnerowicz} for other signatures.

\vspace*{.2cm}

However, all the examples of \cite{frances_ferrand_obata_lorentz} are locally conformally equivalent, and a problem remains open on the local geometry of compact Lorentzian manifolds, often cited in the literature as \textit{Generalized} or \textit{pseudo-Riemannian Lichnerowicz conjecture}, one of its first appearance is in \cite{dambra_gromov}, Section 7.6.

\begin{conjecture}
If a compact Lorentzian manifold has an essential conformal group, then it is conformally flat.
\end{conjecture}

Recall that a pseudo-Riemannian manifold $(M,g)$ is conformally flat if any point admits a neighborhood $U$ such that $g|_U$ is conformal to a flat metric on $U$. Let us point out that the compactness assumption is necessary since most of the metrics Alekseevsky exhibited in \cite{alekseevsky} (7.3) are not conformally flat.

\vspace*{.2cm}

The main result of this article positively answers this conjecture when the manifold admits an essential action of a \textit{simple} Lie group. By an averaging argument, it can be easily observed that any compact group must act inessentially. Thus, we will deal with actions of non-compact simple Lie groups, and we will especially consider the ``smallest'' ones, namely Lie groups locally isomorphic to $\SL(2,\R)$. Even with the simpleness assumption on the acting group, the situation is still very rich. For instance, all the examples of \cite{frances_ferrand_obata_lorentz} admit an essential action of a Lie group locally isomorphic to $\SL(2,\R)$. 

\vspace*{.2cm}

Following the dichotomy inessential/essential, let us first recall the case of \textit{isometric} actions of $\SL(2,\R)$.

\paragraph{Inessential actions: simple Lie groups of Lorentzian isometries.}

Contrarily to Riemannian manifolds, there exists compact Lorentzian manifolds whose isometry group is non-compact. Furthermore, it is possible that the isometry group contains a non-compact simple subgroup. Indeed, consider $H$ a Lie group locally isomorphic to $\SL(2,\R)$ and note $g_K$ its Killing metric. This metric is Lorentzian and invariant under left and right translations of $H$ on itself. Thus, it induces a Lorentzian metric $g$ on any quotient $M := H/\Gamma$ where $\Gamma$ is a uniform lattice of $H$. Since the left action preserves $g_K$ and commutes with the right action, it induces an isometric action of $H$ on $(M,g)$.

\vspace*{.2cm}

As Zimmer first observed in \cite{zimmer86}, such a situation is singular in the sense that up to finite covers, $\PSL(2,\R)$ is the only non-compact simple Lie group that can act faithfully and isometrically on a compact Lorentzian manifold. Deeper in the description, Gromov considered in \cite{gromov} the geometry of a compact Lorentzian manifold $(M,g)$ admitting an isometric action of a Lie group $H$ locally isomorphic to $\SL(2,\R)$. He proved that some isometric cover of $M$ is isometric to a warped product $(H \! ~_{\omega} \! \! \times N)$, where $H$ is endowed with its Killing metric, $N$ is a Riemannian manifold and $\omega : N \rightarrow \R_{>0}$ is a smooth function.

\vspace*{.2cm}

Finally, the situation for isometric actions of non-compact simple Lie group is very rigid and well understood. We now consider essential actions.

\paragraph*{Essential conformal actions of simple Lie groups.}

This subject had been previously investigated in any signature, when the group that acts has \textit{high real-rank}.

\vspace*{.2cm}

In \cite{zimmer87}, Zimmer proved that if a semi-simple Lie group without compact factor acts on a compact pseudo-Riemannian manifold of signature $(p,q)$, with $p \leq q$, then its real-rank is at most $p+1$. In \cite{bader_nevo}, Bader and Nevo proved that if the group is simple and has maximal rank, then it is locally isomorphic to $\SO(p+1,k)$ with $p+1 \leq k \leq q+1$. At last, in \cite{frances_zeghib}, Frances and Zeghib proved that in the same situation, the manifold must be some quotient of the universal cover of the \textit{model space} $\Ein^{p,q}$ of conformal geometry of signature $(p,q)$, introduced in Section \ref{sss:equivalence_principle}. See also \cite{bader_frances_melnick} for analogous results in other parabolic geometries.

\vspace*{.2cm}

Assuming the real-rank maximal restricts a lot the possibilities for the geometry, and a larger variety of examples appears when this assumption is removed, even in Lorentzian signature. As we recalled above, there exists infinitely many compact Lorentzian manifolds with a conformal essential action of a Lie group locally isomorphic to $\SL(2,\R)$, and it seems not plausible to classify these manifolds up to \textit{global} conformal equivalence (\cite{frances_ferrand_obata_lorentz}). However, the dynamics of such a group has implications on the \textit{local} geometry, and it is the main result of this article.

\begin{theorem}
\label{thm:main}
Let $(M^n,g)$, $n \geq 3$, be a smooth compact connected Lorentzian manifold, and $H$ be a connected Lie group locally isomorphic to $\SL(2,\R)$. If $H$ acts conformally and essentially on $(M,g)$, then $(M,g)$ is conformally flat.
\end{theorem}

Since $\sl(2,\R)$ is the most elementary non-compact simple real Lie algebra, it will not be difficult to observe that this theorem positively answers Generalized Lichnerowicz conjecture as soon as the conformal group of the manifold contains a non-compact simple immersed Lie subgroup.

\begin{corollary}
\label{cor:identity_component}
Let $(M^n,g)$ be a smooth compact connected Lorentzian manifold, with $n \geq 3$, and let $G$ be the identity component of its conformal group. Assume that $\g$ contains a non-compact simple Lie subalgebra. If $G$ is essential, then $(M,g)$ is conformally flat.
\end{corollary}

In particular, if a compact connected Lorentzian manifold admits a conformal essential action of a connected semi-simple Lie group, then it is conformally flat.

\paragraph*{The identity component of the conformal group.} Zimmer's result about simple Lie groups of Lorentzian isometries led to the full classification, up to local isomorphism, of the identity component of the isometry group of a compact Lorentzian manifold by Adams, Stuck \cite{adams_stuck,adams_stuckII}, and - independently - Zeghib \cite{zeghib98}. As explained below, Theorem \ref{thm:main} is also a first step in the direction of an analogous classification for the conformal group of a compact Lorentzian manifold.

\vspace*{.2cm}

The Möbius sphere has an analogous object in non-Riemannian conformal geometry: the Einstein Universe $\Ein^{p,q}$ of signature $(p,q)$ (see Section \ref{sss:equivalence_principle}). It is a compact projective quadric, naturally endowed with a conformal class of conformally flat metrics of signature $(p,q)$. Its conformal group is isomorphic to $\PO(p+1,q+1)$ and acts transitively on it.

By a generalization of Liouville's Theorem, if a Lie group $H$ acts on a conformally flat pseudo-Riemannian manifold of signature $(p,q)$, then its Lie algebra $\h$ can be identified with a Lie algebra of conformal vector fields of $\Ein^{p,q}$. In particular, $H$ can be locally embedded into $\PO(p+1,q+1)$.

\vspace*{.2cm}

Thus, by Corollary \ref{cor:identity_component}, if $(M,g)$ is a compact Lorentzian manifold of dimension at least $3$ and if $G$ is the identity component of its conformal group, then we have three possibilities for $G$:
\begin{enumerate}
\item It is inessential, and necessarily belongs to the list of Adams-Stuck-Zeghib classification.
\item It is essential and contains a non-compact simple Lie subgroup, and necessarily it is locally isomorphic to a Lie subgroup of $\SO(2,n)$ since it acts on a conformally flat Lorentzian manifold.
\item It is essential and does not contain non-compact simple Lie subgroups, and by the Levi decomposition, its Lie algebra has the form $\g \simeq \mathfrak{k} \ltimes \text{rad}(\g)$ where $\mathfrak{k}$ is a compact semi-simple Lie algebra and $\text{rad}(\g)$ is the solvable radical of $\g$.
\end{enumerate}
In upcoming works, we will establish that if $\text{rad}(\g)$ has a non-Abelian nilradical, then $G$ is either inessential or locally isomorphic to a subgroup of $\SO(2,n)$ (see \cite{these}, Ch.7 for partial results). 

\vspace*{.2cm}

This suggests that essential conformal groups can always be locally embedded into $\SO(2,n)$. The next important question is to determine which Lie subgroup of $\SO(2,n)$ can exactly be realized as the conformal group of a compact Lorentzian manifold (compare with \cite{adams_stuckII} and Theorem 1.1 of \cite{zeghib98}). 

\paragraph*{Completeness of the associated $(G,\X)$-structure.} A conformally flat pseudo-Riemannian metric of signature $(p,q)$ naturally defines an atlas of $(G,\X)$-manifold, where $\X = \tilde{\Ein}^{p,q}$ and $G =  \tilde{\SO}(p+1,q+1)$. Thus, if a non-compact simple Lie group acts conformally essentially on a compact Lorentzian manifold, then it acts by automorphisms of the associated $(G,\X)$-manifold. When the group is too small, the $(G,\X)$-structure may not be complete.

Indeed, if $k \geq 2$, consider $\R^{1,k}$ the $(k+1)$-dimensional Minkowski space and $\Gamma = <2 \id>$ the group generated by a non-trivial homothety. Naturally, $\Gamma$ acts properly discontinuously and conformally on $\R^{1,k} \setminus \{0\}$ and is centralized by the linear action of $\SO(1,k)$. Therefore, $\SO(1,k)$ acts conformally on the quotient $(\R^{1,k} \setminus \{0\}) / \Gamma$, usually called a \textit{Hopf manifold}. It is a compact conformally flat Lorentzian manifold, whose associated $(G,\X)$-structure is non-complete. Nevertheless, the structure must be complete when other non-compact simple Lie groups act.

\vspace*{.2cm}

Let $(M,g)$ be an $n$-dimensional compact Lorentzian manifold and $G = \Conf(M,g)_0$. If $G$ is essential, then its semi-simple Levi factor is either compact, or locally isomorphic to a Lie subgroup of $\SO(2,n)$. In particular, we recover the main result of \cite{article2}, where we classified semi-simple Lie groups without compact factor that can act conformally on a compact Lorentzian manifold. Up to local isomorphism, the possible groups are
\begin{enumerate}
\item $\SO(1,k)$, $2 \leq k \leq n$;
\item $\SU(1,k)$, $2 \leq k \leq n/2$;
\item $\SO(2,k)$, $2 \leq k \leq n$;
\item $\SO(1,k) \times \SO(1,k')$, $k,k' \geq 2$, $k+k' \leq \max(n,4)$.
\end{enumerate}
Theorem 3 of \cite{frances_zeghib} asserts that when a Lie group locally isomorphic to $\SO(2,k)$ is contained in $G$, then $(M,g)$ is, up to finite cover, a quotient of the universal cover of $\Ein^{1,n-1}$ by an infinite cyclic subgroup of $\tilde{\SO}(2,n)$. The same conclusion can be derived from Theorem 1.5 of \cite{bader_frances_melnick} when we consider actions of $\SO(1,k) \times \SO(1,k')$. An easy consequence of the main result of the present article is that this observation is still valid for $\SU(1,k)$.

\begin{corollary}
\label{cor:su1k}
Let $H$ be a Lie group locally isomorphic to $\SU(1,k)$, $k \geq 2$. Assume that $H$ acts conformally on a smooth compact connected Lorentzian manifold $(M^n,g)$, with $n \geq 3$. Then, $(M,g)$ is conformally diffeomorphic to a quotient $\Gamma \setminus \tilde{\Ein}^{1,n-1}$, where $\Gamma < \tilde{\SO}(2,n)$ is a discrete group acting properly discontinuously on $\tilde{\Ein}^{1,n-1}$.
\end{corollary}

The proof is very short: By Corollary \ref{cor:identity_component}, $(M,g)$ is conformally flat and we can imitate the end of the proof of Theorem 3 of \cite{frances_zeghib}. According to Section 2.4 of this article, it is enough to establish that if $\iota : \su(1,k) \hookrightarrow \so(2,n)$ is a Lie algebra embedding, then the centralizer in $\SO(2,n)$ of the image of $\iota$ is a compact subgroup of $\SO(2,n)$. This can be observed by elementary considerations, that we postpone in an Appendix at the end of the article.

\subsection*{Organization of the article}

Corollary \ref{cor:identity_component} is established in Section \ref{s:inessential}. Precisely, we will prove that as soon as $G$ contains an immersed Lie subgroup $H$ locally isomorphic to $\SL(2,\R)$, $G$ is essential if and only if $H$ is essential. Once it is proved, our problematic is reduced to conformal essential actions of such $H$'s.

In Section \ref{s:minimal_subsets}, we establish a dynamical property of essential conformal actions. By a result of \cite{article3}, $H$ is essential if and only if it does not act everywhere locally freely. We are now going further and describe minimal closed invariant subsets of the action, inside the subset where the action is not locally free, noted $F_{\leq 2}$. The problem is essentially to prove that if a minimal subset contains exclusively $2$-dimensional orbits, then it is in fact a single closed orbit of dimension $2$, which we call \textit{compact conical}. Quickly, this question will be reduced to prove that the flow generated by an hyperbolic one parameter subgroup of $H$ has a periodic orbit. It will be treated by using Osedelec decomposition and general arguments in non-uniformly hyperbolic dynamics.

Conformal flatness of $M$ is then established in two times. Firstly, we will prove in Section \ref{s:small_dimension} that the minimal subsets of $F_{\leq 2}$ previously described admit a conformally flat neighborhood. It is inspired by previous methods (notably \cite{frances_causal_fields}, \cite{frances_zeghib}, \cite{frances_melnick13} and \cite{article3}). Immediately, we will obtain that $F_{\leq 2}$ is contained in a conformally flat open set. Secondly, we will prove in Section \ref{s:extending_flatness} that any $H$-orbit contains a point of $F_{\leq 2}$ in its closure. This dynamical observation will directly extend conformal flatness to the whole manifold.

\subsection*{Conventions}

In this article, $M$ everywhere denotes a connected smooth manifold whose dimension is \textbf{greater than or equal to $3$}.

We note $\mathfrak{X}(M)$ the Lie algebra of vector fields defined on $M$. If $M$ is endowed with a pseudo-Riemannian metric $g$, we note $\Kill(M,[g])$ the Lie algebra of \textbf{conformal Killing vector fields} of $M$, \textit{i.e.} infinitesimal generators of conformal diffeomorphisms. The hypothesis $\dim M \geq 3$ implies that $\Kill(M,[g])$ is always finite dimensional. 

Given a differentiable action of a Lie group $G$ on $M$, we will implicitly identify its Lie algebra $\g$ with a Lie subalgebra of $\mathfrak{X}(M)$ via $X \mapsto \{ \left . \frac{\d}{\d t}\right |_{t=0} e^{-tX}.x\}_{x \in M}$.

We call \textbf{$\sl(2)$-triple} of a Lie algebra any non-zero triple $(X,Y,Z)$ in this Lie algebra satisfying the relations $[X,Y] = Y$, $[X,Z] = -Z$ and $[Y,Z] = X$.

If $f$ is a conformal transformation of $(M,g)$, the function $\varphi	: M \rightarrow \R_{>0}$ such that $f^*g = \varphi g$  is called the \textbf{conformal distortion} of $f$ with respect to $g$. If $\phi^t$ is a conformal flow, its conformal distortion is a cocycle $\lambda : M \times \R \rightarrow \R_{>0}$ over $\phi^t$, such that $[(\phi^t)^* g]_x = \lambda(x,t)g_x$ for all $x \in M$ and $t \in \R$.

If $\dim M \geq 4$, $(M,g)$ is conformally flat if and only if its Weyl tensor $W$ vanishes identically. If $\dim M = 3$, $W$ always vanishes, regardless $(M,g)$ is conformally flat or not. In this situation, conformal flatness is detected by the Cotton tensor of $(M,g)$. In this article, by ``\textbf{Weyl-Cotton curvature}'', we mean the Weyl tensor or the Cotton tensor, depending on whether $\dim M \geq 4$ or not. This tensor will always be noted $W$.

\paragraph*{Acknowledgements.} \textit{
I would like to thank Sylvain Crovisier for suggesting me the use of Pesin Theory in the study of a conformal flow. I am also grateful to Thierry Barbot, Yves Benoist, Charles Frances and Abdelghani Zeghib for useful conversations around this project.
}

\section{Inessential conformal groups}
\label{s:inessential}
Isometric actions of non-compact simple Lie groups on compact Lorentzian manifolds are very well described since the works of Zimmer and Gromov. As we recalled in the introduction, if $H$ is a non-compact simple Lie group, acting by isometries on $(M,g)$, Lorentzian compact, then $H$ is a finite cover of $\PSL(2,\R)$. Moreover, $H$ acts locally freely everywhere and the metric of $M$ induces on every orbit $H.x$ a metric proportional to the image of the Killing metric of $H$ by the orbital map. At last, the distribution orthogonal to the orbits is integrable, with geodesic leaves, proving that some isometric cover of $(M,g)$ is isometric to a warped product $H ~_{\omega} \times N$, with $N$ a Riemannian manifold, $\omega : N \rightarrow \R_{>0}$ and $H$ endowed with its Killing metric.

\vspace*{.2cm}

As it can be easily observed, there are more examples of conformal actions of non-compact simple Lie groups on compact Lorentzian manifolds, \textit{e.g.} simple Lie subgroups of $\PO(2,n)$ acting on $\Ein^{1,n-1}$. If they are not isomorphic to a finite cover of $\PSL(2,\R)$, then they necessarily act essentially. In the remaining cases, we have:

\begin{proposition}[\cite{article3}]
\label{prop:locally_free_sl2_actions}
Let $H$ be a connected Lie group locally isomorphic to $\SL(2,\R)$ and $(M,g)$ be a compact Lorentzian manifold on which $H$ acts conformally. Then, $H$ is inessential if and only if $H$ acts everywhere locally freely.
\end{proposition}

The aim of this first section is to improve this statement. Precisely, we will see that, when they exist, conformal actions of Lie groups locally isomorphic to $\SL(2,\R)$ characterize the essentiality of the full identity component of the conformal group. Coupled with the conclusion of Theorem \ref{thm:main}, this observation will directly give Corollary \ref{cor:identity_component}.

\vspace*{.2cm}

Recall the following fact.

\begin{lemma}[\cite{obata70}, Theorem 2.4]
\label{lem:centralizer_inessential}
Let $(M,g)$ be a pseudo-Riemannian manifold and $X \in \Kill(M,[g])$ be a conformal vector field. If $X$ is nowhere light-like, then $\forall f \in \Conf(M,g)$ such that $f^* X = X$, we have $f \in \Isom(M, \frac{g}{|g(X,X)|})$.
\end{lemma}

The arguments of the proof of Proposition 2.1 of \cite{article3} give the following lemma, that will be reused later in this article.

\begin{lemma}
\label{lem:couple_inessential}
Let $X$ and $Y$ be two complete conformal vector fields of a pseudo-Riemannian manifold $(M,g)$, satisfying $[X,Y] = \lambda Y$ for $\lambda \in \R$ and $g(X,X) > 0$. Let $g_0 := g / g(X,X)$. If the functions $g_0(Y,Y)$ and $g_0(X,Y)$ are bounded along the orbits of $\phi_X^t$, then $X$ and $Y$ are Killing vector fields of $g_0$ and $Y$ is everywhere light-like and orthogonal to $X$.
\end{lemma}

\begin{proof}
Replacing $X$ by $X/\lambda$ if necessary, we can assume that $\lambda \in \{0,1\}$. We still note $g$ the renormalized metric $g / g(X,X)$ (to clarify notations). In any case, since $X$ is preserved by the flow it generates, Lemma \ref{lem:centralizer_inessential} ensures that $\mathcal{L}_X g = 0$.

If $\lambda = 0$, applying Lemma \ref{lem:centralizer_inessential}, we immediately get that $Y$ also preserves $g$.

If $\lambda = 1$, we have $(\phi_{X}^t)_* Y_x = \e^{-t} Y_{\phi_{X}^t(x)}$, and because $\{\phi_{X}^t\} \subset \Isom(M,g)$, we obtain $g_{\phi_{X}^t(x)}(Y,Y) = \e^{2t} g_x(Y,Y)$ and $g_{\phi_{X}^t(x)}(X,Y) = \e^t g_x(X,Y)$. Since we assumed the functions $\{x \mapsto g_x(Y,Y)\}$ and $\{x \mapsto g_x(X,Y)\}$ bounded along any $\phi_X^t$-orbit, we must have $g(Y,Y) = g(X,Y) = 0$ everywhere. Now, the relation $[Y,X] = -Y$ gives $(\phi_ {Y}^t)_* X_x = X_{\phi_ {Y}^t(x)} + t  {Y}_{\phi_ {Y}^t(x)}$. Let $\lambda(x,t)$ be the conformal distortion of $\phi_Y^t$ with respect to $g$. Using that $Y$ is light-like and orthogonal to $X$, we get
\begin{equation*}
\lambda(x,t) g_x(X,X) = g_{\phi_ {Y}^t(x)}(X,X).
\end{equation*}
By construction, the map $\{x \mapsto	g_x(X,X)\}$ is constant equal to $1$. This gives $\lambda(x,t) \equiv 1$, \textit{i.e.} $\phi_ {Y}^t$ is an isometry of $g$.
\end{proof}

\paragraph*{Proof of Corollary \ref{cor:identity_component}.}

Let $(M,g)$ be a compact Lorentzian manifold (recall that we always assume $\dim M \geq 3$) and let $G$ be the identity component of its conformal group. Assume that $G$ contains an immersed Lie subgroup $H \hookrightarrow G$, locally isomorphic to $\SL(2,\R)$. \textit{A priori}, $H$ may not be properly embedded, but we do not need to assume it.

\vspace*{.2cm}

We claim that $G$ is inessential if and only if $H$ is inessential. The non-trivial part of this statement is that if $H$ preserves a metric $g_0$ conformal to $g$, then so does $G$. Let $(X,Y,Z)$ be an $\sl(2)$-triple in $\h$. Since $H$ acts by isometries on $(M,g_0)$, it acts locally freely everywhere and, up to a constant positive factor, the ambient metric induces the Killing metric on the orbits. In particular, the Killing vector field $X$ satisfies $g(X,X) > 0$ everywhere. The adjoint representation $\ad : \h \rightarrow \gl(\g)$ is a representation of $\sl(2,\R)$ on a finite dimensional space. Since $\R.X$ is a Cartan subspace of $\h$, we have that $\ad(X)$ acts diagonally on $\g$. Thus, if $(X_1,\ldots,X_N)$ is a basis of eigenvectors, by compactness of $M$ we can apply Lemma \ref{lem:couple_inessential} to every couples $(X,X_i)$ and conclude that if $g_1$ denotes $g / g(X,X)$, then $\mathcal{L}_{X_i} g_1 = 0$ for all $i$. By connectedness of $G$, we obtain $G = \Isom(M,g_1)_0$.

Corollary \ref{cor:identity_component} is now immediate: if $G$ is essential, then $H$ acts essentially and by Theorem \ref{thm:main}, $(M,g)$ must be conformally flat. We are now reduced to consider conformal essential actions of Lie groups locally isomorphic to $\SL(2,\R)$ on compact Lorentzian manifolds.

\section{Minimal compact subsets of an essential action}
\label{s:minimal_subsets}
In the previous section, we recalled that essential conformal actions are characterized by the fact that they are not everywhere locally free. Naturally, the dynamics in, and near, the closed subset where the action is not locally free plays a central role in the proof of Theorem \ref{thm:main}. This section focuses on its \textit{minimal} compact invariant subsets. 

\vspace*{.2cm}

Precisely, we are now going to establish the first main part of the following proposition, that will be completely proved at the end of the article.

\begin{proposition}
\label{prop:minimal_subsets}
Let $H$ be a connected Lie group locally isomorphic to $\SL(2,\R)$. Assume that $H$ acts conformally and essentially on a compact Lorentzian manifold $(M,g)$. Let $K$ be a minimal $H$-invariant subset. Then, $K$ is either
\begin{enumerate}
\item A global fixed point of the action;
\item Exclusively formed of $1$ dimensional orbits;
\item A compact, positive-degenerate, $2$-dimensional orbit, diffeomorphic to a $2$-torus. This orbit is an homogeneous space of the form
\begin{equation*}
\PSL_k(2,\R) / (\Z \ltimes U),
\end{equation*}
where $\PSL_k(2,\R)$ is the $k$-sheeted cover of $\PSL(2,\R)$, $U$ denotes a unipotent one parameter subgroup and the factor $\Z$ is generated by an element $f$ normalizing $U$ and whose projection in $\PSL(2,\R)$ is hyperbolic.
\end{enumerate}
\end{proposition}

As for a general $\mathcal{C}^1$-action of a Lie group $H$, the map $x \in M \mapsto \dim H.x$ is lower semi-continuous. So, for any $x \in M$ and $y \in \bar{H.x}$, we have $\dim H.y \leq \dim H.x$. This elementary observation implies that all orbits in a minimal compact $H$-invariant subset have the same dimension. If this common dimension is $0$, by connectedness of $H$, $K$ is reduced to a global fixed point. Thus, Proposition \ref{prop:minimal_subsets} essentially says:
\begin{enumerate}
\item There does not exist a compact invariant subset where all orbits have dimension $3$;
\item When all orbits have dimension $2$, $K$ is reduced to the compact orbit of the third point of the proposition.
\end{enumerate} 

We leave in suspense the question of compact invariant subset in the neighborhood of which the action is locally free, their non-existence will be established in Section \ref{s:extending_flatness}. This section is devoted to the proof of the second point. Before starting the proof, let us describe this $2$-dimensional orbit more geometrically.

\subsection{Compact conical orbits of $\PSL(2,\R)$}

Consider the linear action of $\SO_0(1,2) \simeq \PSL(2,\R)$ on the $3$-dimensional Minkowski space $\R^{1,2}$. It acts transitively on the future nullcone $\mathcal{N}^+ = \{(x_1,x_2,x_3) \ | \ x_1^2 = x_2^2 + x_3^2, \ x_1 > 0\}$. Consider now the \textit{Hopf manifold} $(M,g) := (\R^{1,2} \setminus \{0\}) / <\lambda \id>$, $\lambda > 1$. Since the homothety $\lambda \id$ acts conformally on $\R^{1,2}$ and is centralized by $\SO(1,2)$, the latter acts conformally and faithfully on the quotient manifold. In particular, the projection of the nullcone $\mathcal{N}^+ / <\lambda \id>$ is an orbit of $\PSL(2,\R)$, conformally diffeomorphic to $\S^1 \times \S^1$ with the non-negative degenerate metric $\d x_1^2$ (if $x_1$ is the coordinate on the first factor $\S^1$).

If $v \in \mathcal{N}^+$, let $[v]$ denote its projection in the Hopf manifold. The stabilizer of $[v]$ is the group of elements of $\SO_0(1,2)$ preserving $\{\lambda^nv, \ n \in \Z\}$. So, it is included in the stabilizer of the line $\R.v$, which is isomorphic to the affine group $A^+ U < \SO_0(1,2)$, where in a suitable basis of $\R^{1,2}$ starting by $v$, we note
\begin{equation*}
A^+ =
\left \{
\begin{pmatrix}
e^t & & \\
 & 1 & \\
 & & e^{-t}
\end{pmatrix}
, \ t \in \R
\right \}
\text{ and }
U =
\left \{
\begin{pmatrix}
1&t&-t^2/2 \\
0 & 1 & -t \\
0 & 0 & 1
\end{pmatrix}
, \ t \in \R
\right \}.
\end{equation*}
So, the stabilizer of $[v]$ in $\SO_0(1,2)$ is the semi-direct product $< \! \! f \! \! > \! \ltimes \, U$ where
\begin{equation*}
f=
\begin{pmatrix}
\lambda & & \\
& 1 & \\
& & \lambda^{-1}
\end{pmatrix}
\in A^+.
\end{equation*}

Consequently, if $U < \PSL(2,\R)$ is a unipotent one parameter subgroup and if $f \in \PSL(2,\R)$ is hyperbolic an normalizes $U$, we say that $\PSL(2,\R) / (< \! \! f \! \!> \! \ltimes \, U)$ is a \textbf{compact conical} homogeneous space.

More generally, let $H$ be a connected Lie group locally isomorphic to $\SL(2,\R)$, let $\mathcal{Z}$ denote its center and $p : H \rightarrow H/\mathcal{Z} \simeq \PSL(2,\R)$ the natural covering. We say that a homogeneous space $H/H'$ is compact conical if $\mathcal{Z} \cap H'$ has finite index $k \geq 1$ in $\mathcal{Z}$ and $p(H)/p(H')$ is a $\PSL(2,\R)$-compact conical homogeneous space. Note that $H/H'$ is in fact a $\PSL_k(2,\R)$-homogeneous space.

In any event, a compact conical homogeneous space is diffeomorphic to a $2$-torus, homogeneous under some $\PSL_k(2,\R)$, with $k \geq 1$, and it is endowed with the $\PSL_k(2,\R)$-invariant conformal class of non-negative degenerate metrics it inherits from $\mathcal{N}^+$.

\subsection{Proof of Proposition \ref{prop:minimal_subsets} for $2$-dimensional orbits}
\label{ss:proof_2_dimensional_orbits}

Let $(M,g)$ be a compact Lorentzian manifold and $H$ a Lie group locally isomorphic to $\SL(2,\R)$ acting conformally on $(M,g)$. Let $K \subset M$ be a minimal compact $H$-invariant subset such that for all $x \in K$, $\dim H.x = 2$. The aim of this section is to prove that $K$ is a compact conical orbit.

\subsubsection{Tangential information}
\label{sss:tangential_information}

The first step is to observe that the restriction of the ambient metric to any orbit in $K$ is degenerate. To do so, we reuse the following proposition whose proof can be found in \cite{article3}. It is based on the main result of \cite{bader_frances_melnick}, an adaptation of Zimmer's embedding theorem to Cartan geometries.

If $x \in M$, we note $\h_x = \{X \in \h \ | \ X(x) = 0\}$ the Lie algebra of the stabilizer of $x$. Differentiating the orbital map $H \rightarrow H.x$, we obtain a natural identification $T_x(H.x) \simeq \h / \h_x$, so that $\h / \h_x$ inherits a quadratic form $q_x$ from the ambient metric $g_x$.

Let $S<H$ be either an hyperbolic or parabolic one-parameter subgroup, or a connected Lie subgroup whose Lie algebra is isomorphic to the affine algebra $\aff(\R)$. In fact, $S$ is chosen this way because firstly, such groups are amenable, so that for every compact $S$-invariant subset $K \subset M$, there automatically exists an $S$-invariant finite measure whose support is contained in $K$, and secondly, the Zariski closure of $\Ad_{\h}(S)$ in $\GL(\h)$ does not contain any proper algebraic cocompact subgroup. This ensures that we are in the field of application of Theorem 4.1 of \cite{bader_frances_melnick}.

\begin{proposition}[\cite{article3}, Prop. 2.2]
\label{prop:tangential_information}
Let $S<H$ be a subgroup as above. Every closed $S$-invariant subset $F$ contains a point $x$ such that $\Ad(S) \h_x \subset \h_x$ and the induced action $\bar{\Ad}(S)$ on $\h / \h_x$ is conformal with respect to $q_x$.
\end{proposition}

If $S$ is chosen to be a connected Lie subgroup of $H$ locally isomorphic to $\Aff(\R)$, and if we apply Proposition \ref{prop:tangential_information} to $K$, we obtain a point $x_0 \in K$ satisfying the conclusions of the proposition. Let $(X,Y,Z)$ be an $\sl(2)$-triple of $\h$ such that $\s = \Span(X,Y)$. Since $\h_{x_0}$ is an $\ad(\s)$-invariant line of $\h$, it must be $\R.Y$. Thus, the adjoint action of $e^{tY}$ on $\h / \h_{x_0}$ is given in the basis $(\bar{Z},\bar{X})$ by
\begin{equation*}
\begin{pmatrix}
1 & 0 \\
t & 1
\end{pmatrix}.
\end{equation*}
This action being conformal with respect to $q_{x_0}$, we then have $\lambda \in \R$ such that $q_{x_0}(\bar{Z}+t\bar{X}) = e^{\lambda t} q_{x_0}(\bar{Z})$. Since $q_{x_0}$ is the restriction of a Lorentzian metric, it does not vanish identically, implying that $q_{x_0}(Z) \neq 0$, and then $\lambda=0$ since $q_{x_0}(\bar{Z}+t\bar{X})$ is polynomial in $t$. So, $\bar{X}$ is isotropic and orthogonal to $\bar{Z}$ with respect to $q_{x_0}$. This proves that $H.x_0$ is degenerate and that $X_{x_0}$ gives the direction of the kernel at $x_0$, implying that $g_{x_0}(Z,Z) > 0$.

\subsubsection{Stabilizer of $x_0$}

Let $\mathcal{Z}$ be the center of $H$ and let $H_{x_0}$ denote the stabilizer of $x_0$. Note $U < H$ and $A^+<H$ the one-parameter subgroups generated by $Y$ and $X$ respectively, so that $(H_{x_0})_0 = U$. In fact, modulo $\mathcal{Z}$, there are only two subgroups of $H$ admitting $U$ as neutral component. To see this, consider the morphism $\Ad : H \rightarrow \Ad(H) \simeq H / \mathcal{Z} \simeq \SO_0(1,2)$, the last identification coming from the Killing form of $\h$. It is injective in restriction to $A^+U$. The image $\Ad(H_{x_0})$ preserves the line $\R.Y \subset \h$, which is isotropic with respect to the Killing form of $\h$. Thus, $H_{x_0}$ is sent into the stabilizer of $\R.Y$, which is
\begin{equation*}
\Ad(A^+U) \simeq 
\left \{
\begin{pmatrix}
a & au & -au^2/2 \\
0 & 1 & -u \\
0 & 0 & a^{-1}
\end{pmatrix}
, \ a > 0, \ u \in \R
\right \}
\subset \SO_0(1,2)
\end{equation*}
Because $\dim H_{x_0} =1$, $\Ad(H_{x_0})/\Ad(U)$ is either trivial or isomorphic to $\Z$, since it is closed in $\Ad(A^+U)/\Ad(U)$.

\vspace*{.2cm}

Finally, $H_{x_0}/\mathcal{Z}$ is either isomorphic to $U$ or to a semi-direct product $\Z \ltimes U$, where $\Z$ is a discrete subgroup of $A^+$. The issue is to exclude the first case. Otherwise stated, we want to prove the existence of $t_0 > 0$ such that $\phi_X^{t_0}(x_0) = x_0$, \textit{i.e.} that the orbit of $x_0$ under the flow $\phi_X^t$ is periodic. To do so, we are going to prove that this flow is non-uniformly hyperbolic over a compact subset containing $x_0$, with non-zero Lyapunov exponents having all the same sign - except of course the direction of the flow. General arguments based on Pesin Theory will then give the existence of a closed orbit of $\phi_X^t$.

\subsubsection{A lemma on non-uniformly hyperbolic conformal flows}

Let $x_0$ denote the point we have exhibited previously. We define the compact $\phi_X^t$-invariant subset 
\begin{equation*}
K_0 := \overline{\{\phi_X^t(x_0), \ t \in \R \}}.
\end{equation*}
Since we have the general relation $(\phi_X^t)_* Y_x = e^{-t} Y_{\phi_X^t(x)}$ and because $Y_{x_0} = 0$, the vector field $Y$ vanishes on $K_0$. Since $K_0 \subset K$, it implies that the vector fields $X$ and $Z$ are linearly independent in a neighborhood of $K_0$. Moreover, the analogous relation $(\phi_X^t)_* Z_x = e^t Z_{\phi_X^t(x)}$ and the fact that $g_{x_0}(X,X) = g_{x_0}(X,Z) = 0$ implies that $X$ is isotropic and orthogonal to $Z$ everywhere in $K_0$ (since $\phi_X^t$ is conformal). Because $X$ and $Z$ are non-proportional, we get that $g_x(Z,Z) > 0$ for all $x \in K_0$ and by continuity, we have $g(Z,Z) > 0$ in a neighborhood of $K_0$. Let us note
\begin{equation*}
\Omega := \{x \in M \ | \ g_x(Z,Z) > 0 \text{ and } X_x \neq 0\}.
\end{equation*}
In the open subset $\Omega$, we note $g_0 := g/g(Z,Z)$. Consider now the Lorentzian manifold $(\Omega,g_0)$: it is preserved by $\phi_X^t$ - even though it is \textit{not} $H$-invariant - and $K_0 \subset \Omega$ is a compact $\phi_X^t$-invariant subset. Moreover, $X$ is an essential \textit{homothetic} conformal vector field of $(\Omega,g_0)$. Indeed, if $\lambda(x,t)> 0$ is such that $[(\phi_X^t)^*g_0]_x = \lambda(x,t) [g_0]_x$, applying this relation to $Z_x$, we get $e^{2t} = \lambda(x,t)$ for all $x \in \Omega$ and $t \in \R$: the conformal distortion of $\phi_X^t$ is non-trivial and uniform on the manifold.

\begin{lemma}
Let $(M,g)$ be a Lorentzian manifold and $X$ be a complete, non-singular vector field of $(M,g)$ such that $(\phi_X^t)^* g= \e^t g$ for all $t$. Then, any (if any) compact $\phi_X^t$-invariant subset of $M$ is a finite union of light-like periodic orbits of the flow.
\end{lemma}

\begin{proof}
Let $K \subset M$ be a compact $\phi_X^t$-invariant subset, and let $\mu$ be an ergodic $\phi_X^t$-invariant measure such that $\Supp(\mu) \subset K$. We have an Osedelec decomposition $\mu$-almost everywhere $T_xM = E_1(x) \oplus \cdots \oplus E_r(x)$, with Lyapunov exponents $\chi_1 < \cdots < \chi_r$. We claim that $\chi_1=0$, with multiplicity $1$.

By continuity of the Lorentzian metric $g$, for any arbitrary Riemannian norm $\|.\|_x$, there exists $C > 0$ such that for all $x \in K$ and $v \in T_xM$, $|g_x(v,v)| \leq C \|v\|_x^2$ (take for instance $C$ to be the supremum of $|g_x(v,v)|$ over $T^1M|_K$, where $T^1M$ denotes the unit tangent bundle with respect to $\|.\|$). Note $i$ the index such that $\chi_i = 0$ and let $x$ be in the set of full measure where the Osedelec decomposition holds. If $v \in E_1(x) \oplus \cdots E_i(x)$ is non-zero, we have 
\begin{equation*}
\lim_{t \to + \infty} \frac{1}{t} \log \|(\phi_X^t)_* v\|_{\phi_X^t(x)} \leq 0.
\end{equation*}
But on the other hand, $g_{\phi_X^t(x)}((\phi_X^t)_* v,(\phi_X^t)_* v) = \e^t g_x(v,v)$. Since we can compare $g$ and $\|.\|$ over $K$, we obtain $t + \log |g_x(v,v)| \leq \log C + 2\log \|(\phi_X^t)_* v\|_{\phi_X^t(x)}$. Therefore, we must have $g_x(v,v) = 0$, for any $v \in E_1(x) \oplus \cdots E_i(x)$. Since $g$ has Lorentzian signature, its totally isotropic subspaces are at most $1$-dimensional. Thus, we get that $E_1(x) \oplus \cdots \oplus E_i(x)$ is $\mu$-almost everywhere reduced to the direction of the flow, and that this direction is isotropic.

\vspace*{.2cm}

In what follows, we forget about the conformal Lorentzian aspects of our problem and only consider the differentiable dynamics of $\varphi^t := \phi_X^{-t}$ when $t \to +\infty$. We note $d$ a distance induced by a Riemannian norm on $M$. This flow is non-uniformly hyperbolic since the Lyapunov exponent $0$ has multiplicity $1$, all other exponents being negative. So, we are in the setting of Pesin Theory. For any $\lambda \in ]0,\chi_2[$, it gives us a set of full measure $\Lambda$ and for all $x \in \Lambda$, a local \textit{stable manifold} $W_{\text{loc}}^s(x)$ of codimension $1$ since there are no expanding directions, \cite{barreira_pesin} Theorem 7.7.1. The fundamental property of local stable manifolds that we will use is that there exists $\gamma(x) > 0$ such that for all $y,z \in W_{\text{loc}}^s(x)$ and $t \geq 0$,
\begin{equation}
\label{equ:exponential_rate_varphi}
d(\varphi^t (y), \varphi^t (z)) \leq \gamma(x) d(y,z) e^{-\lambda t}.
\end{equation}
Shrinking $W_{\text{loc}}^s(x)$ if necessary, we can assume that (\ref{equ:exponential_rate_varphi}) holds for $y$ and $z$ in the closure of $W_{\text{loc}}^s(x)$ and that $W_{\text{loc}}^s(x)$ is transverse to the flow, so that we have $\epsilon(x)>0$ such that $(t,y) \in ]-\epsilon(x),\epsilon(x)[ \times W_{\text{loc}}^s(x) \mapsto \varphi^t(y)$ is a diffeomorphism onto its image $B_x^{\epsilon(x)}$, called a \textit{flow box} at $x$.

By the Poincaré recurrence theorem, $\Lambda \cap K$ contains recurrent points for $\phi_X^t$. Let $x$ be one of them. Let $\delta >0$ such that $B(x,\delta) \subset B_x^{\epsilon(x)}$ - where $B(x,\delta)$ is the ball of radius $\delta$ with respect to $d$. Since $x$ is recurrent, we have $T>0$, as big as we want, such that $\varphi^T(x) \in B(x,\delta/2)$. By (\ref{equ:exponential_rate_varphi}), we can also assume that $T$ is such that for all $y \in \bar{W_{\text{loc}}^s(x)}$, we have $d(\varphi^T(x),\varphi^T(y)) < \delta/2$. Thus, $\varphi^T$ maps $\bar{W_{\text{loc}}^s(x)}$ into the flow box. Let $\pi_x : B_x^{\epsilon(x)} \rightarrow W_{\text{loc}}^s(x)$ be the natural submersion obtained by flowing with times not greater than $\epsilon(x)$. Finally, we have a continuous map 
\begin{equation*}
f := \pi_x \circ \varphi^T : \bar{W_{\text{loc}}^s(x)} \rightarrow \bar{W_{\text{loc}}^s(x)}.
\end{equation*}
Since $\pi_x$ is obtained by flowing in a small region, it is a Lipschitz map. So, replacing $T$ by a greater value if necessary and using (\ref{equ:exponential_rate_varphi}), we get that $f$ is a contraction map. The Picard fixed-point Theorem applies and gives a fixed point $x' \in \bar{W_{\text{loc}}^s(x)}$. This means that $\varphi^{T+t}(x') = x'$ for some $t \in ]-\epsilon(x),\epsilon(x)[$: we have found a periodic orbit of the flow. We claim that moreover, $x \in \mathcal{O}_{x'} := \{\varphi^t(x'), \ t\in \R\}$. Indeed, we have $d(\varphi^t(x),\varphi^t(x')) \rightarrow 0$ and $x$ is a recurrent point. It implies that $d(x,\mathcal{O}_{x'})=0$, and then $x \in \mathcal{O}_{x'}$ since $\mathcal{O}_{x'}$ is compact.

This proves in particular that any minimal $\varphi^t$-invariant subset of $K$ is a periodic orbit. It is not difficult to see that in fact, any point of $K$ has a periodic orbit. Indeed, if $x \in K$ consider the $\alpha$-limit set $\alpha(x) = \cap_{t \in \R} \overline{\{ \varphi^s(x), \ s \leq t\}}$. What we have seen above ensures that some point $x^- \in \alpha(x)$ has a periodic orbit and a stable codimension $1$ manifold $W_{\text{loc}}^s(x^-)$, satisfying (\ref{equ:exponential_rate_varphi}). Thus, $x^-$ admits a neighborhood $V$ such that there exists $C \geq 0$ such that for any $y \in V$, there is $t(y) \in \R$ such that for all $t \geq 0$
\begin{equation*}
d(\varphi^t(y),\varphi^{t+t(y)}(x^-)) \leq C e^{-\lambda t}.
\end{equation*}
Let $\mathcal{O}_{x^-}$ denote the orbit of $x^-$. Let $t_n \to +\infty$ be a sequence such that $y_n= \varphi^{-t_n}(x) \rightarrow x^-$. If $n$ is large enough, $y_n \in V$. So, $d(x,\mathcal{O}_{x^-}) \leq d(\varphi^{t_n}(y_n),\varphi^{t_n+t(y_n)}(x^-)) \leq C e^{-\lambda t_n} $. This proves $d(x,\mathcal{O}_{x^-}) = 0$, \textit{i.e.} $x$ belongs to the orbit of $x^-$.

Finally, the same argument gives that if $x \in K$, then $x$ admits a neighborhood $V$ such that $V \cap K = V \cap \{\varphi^t(x), \ t \in \R\}$. By compactness, $K$ contains a finite number of periodic orbits.
\end{proof}

\subsubsection{Conclusion}

If we apply this result to $(\Omega,g_0)$ with the homothetic action of $\phi_X^t$, we obtain that $K_0$ is in fact reduced to a periodic orbit of $x_0$. Thus, we have $t_0 > 0$ such that $\phi_X^{t_0}(x_0) = x_0$, \textit{i.e.} $H_{x_0} \cap A^+ \neq \{\id\}$. So, $H_{x_0}/ \mathcal{Z} \simeq \Z \ltimes U$. In particular, if the center $\mathcal{Z}$ is finite, the orbit is compact conical and we are done as soon as $H \neq \tilde{\SL}(2,\R)$.

\paragraph*{The case of $\tilde{\SL}(2,\R)$.}

Assume now that $H$ is isomorphic to $\tilde{\SL}(2,\R)$. We still have $H_{x_0} / \mathcal{Z} = \, < \! \! f \! \! > \! \ltimes \: U$ where $f \in \PSL(2,\R)$ is hyperbolic and normalizes the unipotent one-parameter subgroup $U$. Let $\zeta \in \mathcal{Z}$ be a generator. Let $n_k \rightarrow \infty$ be an increasing sequence such that $\zeta^{n_k}(x_0) \rightarrow x$. Since $\zeta$ centralizes $X$, $Y$ and $Z$, and is conformal, we recover at $x$ the same properties as at $x_0$: $\phi_X^{t_0}(x) = x$, $Y_x=0$ and $X_x$ is isotropic and orthogonal to $Z_x$. The same arguments based on local stable manifolds at (or near) $x$ ensures that there is a neighborhood $V$ of $x$ such that, if $\mathcal{O}_x$ denotes the (closed) $\phi_X^t$-orbit of $x$, then for any $y \in V$, $d(\phi_X^t(y),\mathcal{O}_x) \rightarrow 0$ when $t \to -\infty$. But since $\zeta$ centralizes $X$, for any $k$, $\zeta^{n_k}(x_0)$ is a periodic point of $\phi_X^t$. So, if $k$ is such that $\zeta^{n_k}(x_0) \in V$, then the distance between the orbit of $\zeta^{n_k}(x_0)$ and the orbit of $x$ is zero, \textit{i.e.} $\zeta^{n_k}(x_0)$ belongs to the $\phi_X^t$-orbit of $x$ for $k$ large enough. So, for large $k$, we have $t_k$ such that $\zeta^{n_k}(x_0) = \phi_X^{t_k}(x)$. If $p = n_{k+1} - n_k$ and $t = t_k - t_{k+1}$, we obtain $\zeta^p \circ \phi_X^t (x_0) = x_0$, \textit{i.e.} $\zeta^p . e^{tX} \in H_{x_0}$. If we had $H_{x_0} \cap \mathcal{Z} = \{id\}$, then we would have $\zeta^p \in A^+ U$ where $A^+$ and $U$ are the one-parameter subgroups generated by $X$ and $Y$. This is not possible since no element in $A^+ U$ centralizes all $\tilde{\SL}(2,\R)$. So, some power $\zeta^m$ fixes $x_0$, proving that the orbit of $x_0$ is also a compact conical orbit.

\section{Conformal flatness near orbits with small dimension}
\label{s:small_dimension}
A conformal actions of $H \simeq_{\text{loc}} \SL(2,\R)$ on a compact Lorentz manifold $(M,g)$ is essential if and only if there exists an orbit of dimension at most $2$. Let us note
\begin{equation*}
F_{\leq 2} = \{x \in M \ | \ \dim H.x \leq 2\}.
\end{equation*}
It is a non-empty $H$-invariant compact subset of $M$. Considering a minimal $H$-invariant subset of $F_{\leq 2}$, what we have done so far proves that $F_{\leq 2}$ contains either a fixed point, or a $1$-dimensional orbit, or a compact conical orbit. We are now going to prove that such orbits always admit a conformally flat neighborhood.

\vspace*{.2cm}

Except in the first case, the key point is that each time, the isotropy of the orbit contains either an hyperbolic flow, or just an hyperbolic element, whose dynamics imposes that a neighborhood of the orbit is conformally flat. Once the action is described, the vanishing of the Weyl-Cotton curvature easily follows from previous methods (\cite{frances_melnick13}, \cite{article3}). We will determine the dynamics of this hyperbolic flow or element by using the \textit{Cartan geometry} associated to the conformal structure of the manifold.

\vspace*{.2cm}

Let us mention that in the case of a $1$-dimensional orbit and of a compact conical orbit, the techniques involved are local: we make no use of the global action of $H$. In particular, the conclusions are valid in non-compact Lorentzian manifolds.

\subsection{Preliminaries on Cartan geometries}
\label{ss:preliminaries_cartan_geometries}
Let $G$ be a Lie group, $P<G$ a closed subgroup and $n = \dim G/P$.

\begin{definition}
Let $M$ be a differentiable manifold of dimension $n$. A \textit{Cartan geometry} on $M$, with model space $G/P$, is the data of a $P$-principal fiber bundle $\pi : \hat{M} \rightarrow M$, together with a $1$-form $\omega \in \Omega^1(\hat{M},\g)$, such that:
\begin{enumerate}
\item $\forall \hx \in \hat{M}$, $\omega_{\hx} : T\hat{M} \rightarrow \g$ is a linear isomorphism ;
\item $\forall p \in P$, $(R_p)^* \omega = \Ad(p^{-1}) \omega$ ;
\item $\forall A \in \p$, $\omega(A^*) \equiv A$, where $A^*$ denotes the fundamental vector field on $\hat{M}$ associated to the right action of $e^{tA}$.
\end{enumerate}
\end{definition}

The bundle $\pi : \hat{M} \rightarrow M$ is called the \textit{Cartan bundle} and $\omega$ is called the \textit{Cartan connection}. A morphism between two Cartan geometries $(M_1,\hat{M}_1,\omega_1)$ and $(M_2,\hat{M}_2,\omega_2)$ is a local diffeomorphism $f : M_1 \rightarrow M_2$ such that there exists a bundle morphism $\hat{f} : \hat{M}_1 \rightarrow \hat{M}_2$ covering $f$, and such that $\hat{f}^* \omega_2 = \omega_1$. If the model space $G/P$ is effective, a morphism $f$ uniquely determines the bundle morphism $\hat{f}$ covering it (\cite{cap_slovak}, Prop.1.5.3). In such cases, we say that $\hat{f}$ is \textit{the lift} of $f$.

At the infinitesimal level, a vector field $X \in \mathfrak{X}(M)$ is said to be a \textit{Killing vector field} of the Cartan geometry if its local flow is formed with local automorphisms. This is equivalent to the existence of $\hat{X} \in \mathfrak{X}(\hat{M})$ such that $\pi_* \hat{X} = X$, $\forall p \in P$, $(R_p)^* \hat{X} = \hat{X}$ and $\mathcal{L}_{\hat{X}} \omega = 0$. When $G/P$ is effective, we have a well-defined correspondence $X \mapsto \hat{X}$, and $\hat{X}$ is called the lift of $X$.

\paragraph*{Holonomy of a transformation admitting a fixed point.}

Let $f$ be an automorphism of a Cartan geometry and $\hx \in \hat{M}$. If $M$ is connected, then $\hat{f}$, and \textit{a fortiori} $f$, is completely determined by the evaluation $\hat{f}(\hx)$ at $\hx$. If we assume that $f(x) = x$, then $\hat{f}$ preserves the fiber $\pi^{-1}(x) = \hx.P$. In particular, there exists a unique $p \in P$ such that $\hat{f}(\hx) = \hx.p$. Following \cite{frances_localdynamics}, we say that $p$ is the \textit{holonomy of $f$ at $\hx$}. This element $p$ determines $f$ and the principle is that the description of the action of $f$ near $x$ can be reduced to an algebraic analysis of its holonomy.
 
If a Killing vector field $X$ is such that $X(x) = 0$, then $\hat{X}(\hx)$ is tangent to the fiber $\pi^{-1}(x)$, and $X_h := \omega_{\hx}(\hat{X}_{\hx}) \in \p$ is called the holonomy of $X$ at $\hx$. Equivalently, it can be defined by the fact that $e^{tX_h}$ is the holonomy at $\hx$ of $\phi_X^t$, for small $t$.

\subsubsection{The equivalence principle for conformal structures}
\label{sss:equivalence_principle}

\paragraph*{Einstein Universe.} Let $(p,q)$ be two non-negative integers such that $n:=p+q \geq 3$. The Einstein Universe of signature $(p,q)$, noted $\Ein^{p,q}$, is defined as the projectivized nullcone $\mathcal{N}^{p+1,q+1} \setminus \{0\} = \{(x_1,\ldots,x_{n+2}) \in \R^{n+2} \setminus \{0\} \ | \ -x_1^2 - \cdots - x_{p+1}^2 + x_{p+2}^2 + \cdots + x_{n+2}^2 = 0\}$. It is a smooth quadric hypersurface of $\R P^{n+1}$, that naturally inherits a conformal class $[g_{p,q}]$ of signature $(p,q)$ from the ambiant quadratic form of $\R^{p+1,q+1}$. It admits a double cover $\S^p \times \S^q \rightarrow \Ein^{p,q}$. By construction, there is a natural transitive conformal action of $\PO(p+1,q+1)$ on $\Ein^{p,q}$, and in fact $\Conf(\Ein^{p,q},[g_{p,q}]) = \PO(p+1,q+1)$. Thus, $\Ein^{p,q}$ is a compact, conformally homogeneous space. It is \textit{the model space} of conformal geometry in the following sense.

\begin{theorem}[Equivalence principle]
Let $(p,q)$ be a couple of non-negative integers such that $p+q \geq 3$. There is an equivalence of category between the category of conformal structures of signature $(p,q)$ and the category of normalized Cartan geometries modeled on $\Ein^{p,q}$.
\end{theorem}

This result was originally proved by E. Cartan in the Riemannian case. See \cite{sharpe}, Ch. V., and \cite{cap_slovak}, Section 1.6, for references. The normalization condition is an additional technical condition imposed on the Cartan connection of the corresponding Cartan geometry. We do not give detail since it will not be useful for us.

Since $\Ein^{p,q}$, as a $\PO(p+1,q+1)$-homogeneous space, is effective, we can legitimately consider the lifts to the Cartan bundle of conformal maps and conformal vector fields defined on the base manifold. 

\subsubsection{Explicit root-space decomposition of $\so(2,n)$}
\label{sss:cartan_decomposition}

The theory of Cartan geometries allows us to reduce technical problem of conformal geometry to algebraic questions in the model space. From now on, we only consider Lorentzian conformal structures and the letter $G$ exclusively refers to the Lie group $\PO(2,n)$, and $P$ will denote the stabilizer in $G$ of an isotropic line in $\R^{2,n}$, so that $\Ein^{1,n-1} \simeq G/P$ as $G$-homogeneous spaces.

\vspace*{.2cm}

We adopt here some of the notations of \cite{cap_slovak}, Section 1.6.3. In a basis of $\R^{2,n}$ in which the quadratic form reads $2x_1x_{n+2} + 2x_2x_{n+1}+ x_3^2 + \cdots + x_n^2$, and $P$ is the stabilizer of $[1:0:\cdots:0]$, the Lie algebra $\g= \so(2,n)$ has the form
\begin{equation*}
\g= 
\left \{
\begin{pmatrix}
a & Z & 0 \\
X & A & -Z^* \\
0 & -X^* & -a
\end{pmatrix}
, \ a \in \R, \ X \in \R^n, \ Z \in (\R^n)^*, \ A \in \so(1,n-1)
\right \}
\end{equation*}
where $Z^*$ denotes $J ~ ^{t} Z$, $X^* = ~ ^{t} \! X J$ and $JA+ ~^{t}A J = 0$, with
$
J := 
\begin{pmatrix}
0 & 0 & 1 \\
0 & I_{n-2} & 0 \\
1 & 0 & 0
\end{pmatrix}.
$

Abusively, we will write $Z$ (or $X$) to denote the corresponding elements of $\g$. This decomposition yields the grading $\g= \g_{-1} \oplus \g_0 \oplus \g_1$ (see \cite{cap_slovak}, p.118) and we have $\p = \g_0 \oplus \g_1$. Deeper in the description, we can decompose the $\so(1,n-1)$ factor similarly:
\begin{equation*}
\so(1,n-1)= 
\left \{
\begin{pmatrix}
b & T & 0 \\
U & B & - ^{t}  T \\
0 & - ^{t} U & -b
\end{pmatrix}
, \ b \in \R, \ U \in \R^{n-2}, \ T \in (\R^{n-2})^*, \ B \in \so(n-2)
\right \}.
\end{equation*}
Then, we identify a Cartan subspace in $\so(2,n)$, with respect to the Cartan involution $\theta(M) = - ^{t} M$:
\begin{equation*}
\a =
\left \{
\begin{pmatrix}
a & & & & \\
  & b & & & \\
  & & 0 & & \\
  & & & -b & \\
  & & & & -a
\end{pmatrix}
,\ a,b \in \R 
\right \}.
\end{equation*}
The corresponding restricted root-space decomposition is summarized below
\begin{equation*}
\begin{pmatrix}
\a & \g_{\alpha} & \g_{\alpha + \beta} & \g_{\alpha+2\beta} & 0 \\
   & \a & \g_{\beta} & 0 & \g_{\alpha+2\beta} \\
   & & \m & \g_{\beta} & \g_{\alpha + \beta}  \\
   & & & \a & \g_{\alpha} \\
   & & & & \a
\end{pmatrix}
\end{equation*}
(the negative root spaces being obtained by transposition). The factor $\m = \mathfrak{z}_{\mathfrak{k}}(\a)$ is isomorphic to $\so(n-2)$ and corresponds to the block matrices $B$ in the decomposition of $\so(1,n-1)$. The simple roots $\alpha$ and $\beta$ are given by $\alpha(a,b) = a-b$ and $\beta(a,b) = b$, where $(a,b)$ abusively refers to the corresponding matrix of $\a$. The root spaces $\g_{\pm \beta}$ and $\g_{\pm (\alpha + \beta)}$ have dimension $n-2$, while $\g_{\pm \alpha}$ and $\g_{\pm (\alpha + 2\beta)}$ are lines. We have $\g_1 = \g_{\alpha} \oplus \g_{\alpha + \beta} \oplus \g_{\alpha + 2\beta}$.

\subsection{$1$-dimensional orbits}
\label{ss:1_dimensional_orbits}
Let $H$ be a Lie group locally isomorphic to $\SL(2,\R)$ and $(M,g)$ a Lorentzian manifold on which $H$ acts conformally. We assume in this section that there exists a $1$-dimensional orbit $H.x_0$ in $M$. The stabilizer $\h_{x_0}$ is a $2$-dimensional subalgebra of $\sl(2,\R)$. So, it must be isomorphic to the affine algebra and there exists an $\sl(2)$-triple $(X,Y,Z)$ such that $\h_{x_0} = \Span(X,Y)$.

\subsubsection{Holonomy of the stabilizer}

Let $\pi : \hat{M} \rightarrow M$ and $\omega \in \Omega^1(\hat{M},\g)$ denote the Cartan bundle and the Cartan connection defined by the conformal class $[g]$. If $A$ is a conformal vector field vanishing at a point $x$ and $\hx \in \pi^{-1}(x)$, its holonomy at $\hx$, noted $A_h \in \p$, determines the behaviour of $A$ near its singularity $x$. However, it is complicated to relate explicitly, in full generality, the dynamics of $A$ near $x$ to the algebraic properties of $A_h$.

We start here by describing the holonomies of $X$ and $Y$. Since we have here an $\sl(2)$-triple of conformal vector fields, this question will essentially be reduced to a classification of morphisms $\sl(2,\R) \rightarrow \so(2,n)$.

\vspace*{.2cm}

Let $\hx_0 \in \pi^{-1}(x_0)$ and let $X_h$ and $Y_h$ denote the holonomies of $X$ and $Y$ at $\hx_0$. Remark that a different choice of $\hx_0$, say $\hx_0.p$ with $p \in P$, changes $X_h$ and $Y_h$ in $\Ad(p^{-1})X_h$ and $\Ad(p^{-1})Y_h$. Let $Z^* \in \so(2,n)$ denote the element $\omega_{\hx_0}(\hat{Z})$. We claim that $(-X_h,-Y_h,-Z^*)$ is an $\sl(2)$-triple of $\so(2,n)$. To see this, we introduce a central object of Cartan geometries: the curvature form $\Omega := \d \omega + \frac{1}{2}[\omega,\omega] \in \Omega^2(\hat{M},\so(2,n))$. As it is done in \cite{bader_frances_melnick}, Lem. 2.1, we can compute that for all $\hx \in \hat{M}$,
\begin{equation*}
\omega_{\hx}([\hat{A},\hat{B}]) + [\omega_{\hx}(\hat{A}),\omega_{\hx}(\hat{B})] = \Omega_{\hx}(\hat{A},\hat{B}),
\end{equation*}
for any pair of conformal vector fields $(A,B)$. An elementary property of the curvature form is its horizontality: it vanishes as soon as one of its argument is tangent to the fiber of $\hat{M}$ (\cite{sharpe}, Ch.5, Cor. 3.10). Since $X$ and $Y$ vanish at $\hx_0$, their lifts are vertical and the previous formula ensures that $-X_h$, $-Y_h$ and $- Z^*$ satisfy the bracket relations of $\sl(2,\R)$.

\vspace{.2cm}

Thus, we have obtained a representation $\rho : \sl(2,\R) \rightarrow \so(2,n)$ such that $\rho(X)$ and $\rho(Y)$ admit a common isotropic eigenvector $v \in \R^{2,n}$, which is not an eigenvector for $\rho(Z)$. In particular, $v$ is a highest weight vector for $\rho$, and it follows that the subspace $V = \Span(\rho(Z)^k v, \ k \geq 0)$ is a faithful irreducible subrepresentation of $\rho$.

\paragraph*{Orthogonal representations of $\sl(2,\R)$.}

The following property reduces the possibilities for $V$.

\begin{lemma}
\label{lem:sl2_irred_orthogonal}
Let $\pi_d : \sl(2,\R) \rightarrow \gl(V_d)$ be the $(d+1)$-dimensional irreducible representation of $\sl(2,\R)$. Let $Q$ be a non-zero quadratic form on $V_d$ such that $\pi_d(\sl(2,\R)) \subset \so(Q)$. Then, $d$ is even and $Q$ is non-degenerate, with signature $(\frac{d}{2},\frac{d}{2}+1)$ or the opposite. Moreover, $Q$ is uniquely determined up to a multiplicative constant.
\end{lemma}

Since $V$ is a subspace of $\R^{2,n}$ with dimension greater than $1$, we distinguish four possibilities:

\begin{enumerate}
\item \label{case:dim=2} $\dim V=2$ and $V$ is a totally isotropic plane ;
\item \label{case:(1,2)} $\dim V=3$ and has signature $(1,2)$ ;
\item \label{case:(2,1)} $\dim V=3$ and has signature $(2,1)$ ;
\item \label{case:(2,3)} $\dim V=5$ and has signature $(2,3)$.
\end{enumerate}

We now treat each situation separately. We note $Q$ the quadratic form of $\R^{2,n}$. We wish to obtain the form of $\rho(X)$ and $\rho(Y)$, up to conjugacy in $P$, which is the stabilizer of the line $\R.v \subset \R^{2,n}$. So, we will say that a basis $(e_1,\ldots,e_{n+2})$ is adapted if $Q$ reads $2x_1x_{n+2}+2x_2x_{n+1}+x_3^2+ \cdots + x_n^2$ and $e_1=\lambda v$

\paragraph*{Case \ref{case:dim=2}.} The orthogonal $V^{\perp}$ is also a subrepresentation of $\rho$ and $Q$ is non-negative on $V^{\perp}$, with $\Ker(Q|_{V^{\perp}}) = V$. Since $Q|_{V^{\perp}} \geq 0$, Lemma \ref{lem:sl2_irred_orthogonal} ensures that any non-trivial irreducible subrepresentation of $\rho|_{V^{\perp}}$ must be an isotropic plane, \textit{i.e.} must coincide with $V$. Since $\rho|_{V^{\perp}}$ is completely reducible, this means that there exists a subspace $E$ such that $\rho|_E = 0$ and $V^{\perp} = V \oplus E$. Since $E$ is a Euclidean subspace of $\R^{2,n}$, $E^{\perp}$ has signature $(2,2)$ and is also a subrepresentation of $\rho$. If $V'$ is now an isotropic plane such that $E^{\perp} = V \oplus V'$ and if $(e_1,e_2,e_{n+1},e_{n+2})$ is a basis of $E^{\perp}$ adapted to this decomposition, such that $e_1 = v$ and the quadratic form reads $2x_1x_{n+2}+2x_2x_{n+1}$, then $\rho|_{E^{\perp}}$ has the form
\begin{equation*}
2aX+\sqrt 2 bY + \sqrt 2 cZ
\mapsto
\begin{pmatrix}
a & b & 0 & 0 \\
c & -a & 0 & 0 \\
0 & 0 & a & -b \\
0 & 0 & -c & -a
\end{pmatrix}
\in \so(E^{\perp}) \simeq \so(2,2)
\end{equation*}
If we complete this basis with an orthonormal basis of $E$, we obtain an adapted basis of $\R^{2,n}$ in which 
\begin{equation*}
2 \rho(X) = 
\begin{pmatrix}
1 &  & & & & &  \\
  &-1& & & & & \\
  &  &0& & & & \\
  &  & & \ddots & & & \\
  &  & & & 0 & & \\
  &  & & & & 1 & \\
  & & & & & & -1
\end{pmatrix}
\text{ and }
\sqrt 2 \rho(Y) =
\begin{pmatrix}
0 & 1 & 0 & \cdots & 0 & 0 & \\
 & 0 & & & & & \\
 & & 0 & & & & \\
 & & & \ddots & & & \vdots \\
 & & & & 0 & & 0 \\
 & & & & & 0 & -1 \\
 & & & & & & 0
\end{pmatrix}
.
\end{equation*}
\paragraph*{Case \ref{case:(1,2)}.} In this situation, $V^{\perp}$ is Lorentzian and supplementary to $V$. We then have two subcases.

\vspace*{.2cm}

\begin{enumerate}[label=\textbf{\alph*.}]
\item 
\label{case2a}
If $\rho|_{V^{\perp}} = 0$, then $\rho = (\rho|_V,0)$ (orthogonal decomposition). The Lorentzian representation $V$ has the form
\begin{equation*}
aX+bY+cZ
\mapsto
\begin{pmatrix}
a & b & 0 \\
c & 0 & -b \\
0 & -c & -a
\end{pmatrix}
\in \so(V) \simeq \so(1,2)
\end{equation*}
in a basis $(e_1,e_3,e_{n+2})$ such that $e_1 = v$ and the quadratic form reads $2x_1x_{n+2}+x_3^2$. Thus, this basis can be completed into an adapted basis of $\R^{2,n}$ in which we have
\begin{equation*}
\rho(X) = 
\begin{pmatrix}
1 & & & & \\
 & 0 & & & \\
 & & \ddots & & \\
 & & & 0 & \\
  & & & & -1
\end{pmatrix}
\end{equation*}

\item 
\label{case2b}
If $\rho|_{V^{\perp}} \neq 0$, then it is a faithful representation of $\sl(2,\R)$ into $\so(V^{\perp})$ and $V^{\perp}$ is Lorentzian. But up to conjugacy in $O(V^{\perp})$, this representation is unique. Indeed, it admits a non-trivial irreducible subrepresentation $V' \subset V^{\perp}$. By Lemma \ref{lem:sl2_irred_orthogonal}, the only possibility is that this subrepresentation is $3$-dimensional and Lorentzian. So, if $E = (V \oplus V')^{\perp}$, then $E$ is Riemannian and $\rho$-invariant, so $\rho|_E = 0$. Thus, $\rho$ is conjugate to $(\rho|_V,\rho|_{V'},0)$ (orthogonal decomposition). Thus, if $(e_1,e_3,e_{n+2})$ is the same basis of $V$ as in Case \ref{case:(1,2)}.a, if $(e_2,e_4,e_{n+1})$ is a basis of $V'$ in which $\rho|_{V'}$ has the form 
\begin{equation*}
aX+bY+cZ
\mapsto
\begin{pmatrix}
a & b & 0 \\
c & 0 & -b \\
0 & -c & -a
\end{pmatrix}
\in \so(V') \simeq \so(1,2)
\end{equation*}
and if we choose $(e_5,\ldots,e_n)$ an orthonormal basis of $E$, then $(e_1,\ldots,e_{n+2})$ is an adapted basis of $\R^{2,n}$ in which
\begin{equation*}
\rho(X) = 
\begin{pmatrix}
1 & & & & & & \\
 & 1 & & & & & \\
 & & 0 & & & &  \\
 & & & \ddots & & & \\
 & & & & 0 & & \\
 & & & & & -1 &  \\
 & & & & & & -1 
\end{pmatrix}
\text{ and }
\rho(Y) = 
\begin{pmatrix}
0 & 0 & 1 & 0 & 0 & \cdots & 0 \\
 & 0 & 0 & 1 & 0 & \cdots & 0 \\
 & & 0 & & & 0 & -1 \\
 & & & \ddots & & -1 & 0 \\
 & & & & & 0 & 0 \\
 & & & & & \vdots & \vdots \\
 & & & & & & 0
\end{pmatrix}
\end{equation*}

\end{enumerate}

\paragraph*{Case \ref{case:(2,1)}.} In this situation, $V^{\perp}$ is Riemannian. Therefore, $\rho|_{V^{\perp}} = 0$ and we are in a situation similar to Case \ref{case:(1,2)}.a. So, there is an adapted basis of $\R^{2,n}$ in which
\begin{equation*}
\rho(X) = 
\begin{pmatrix}
1 & & & & \\
 & 0 & & & \\
 & & \ddots & & \\
 & & & 0 & \\
 & & & & -1
\end{pmatrix}
\end{equation*}

\paragraph*{Case \ref{case:(2,3)}.} Here, $V^{\perp}$ is Riemannian and $\rho|_{V^{\perp}} = 0$. In this situation, we have an adapted basis such that $(e_1,e_2,e_3,e_{n+1},e_{n+2})$ is a basis of $V$ and
\begin{equation*}
\rho(X)
=
\begin{pmatrix}
2 & & & & & & \\
 & 1 & & & & & \\
 & & 0 & & & & \\
 & & & \ddots & & & \\
 & & & & 0 & & \\
 & & & & & -1 & \\
 & & & & & & -2
\end{pmatrix}
\end{equation*}

\subsubsection{Dynamics of $X$}
\label{sss:dynamics_of_X}

A fundamental property that can be easily read on the holonomy of a conformal transformation $f$ fixing a point $x$ is its \textit{linearizability} near $x$. Note $p$ the holonomy of $f$ at $\hx \in \pi^{-1}(x)$, \textit{i.e.} the unique $p \in P$ such that $\hat{f}(\hx) = \hx .p$. Recall that $P$ can be seen as the (affine) conformal group of $\R^{1,n-1}$, namely $\CO(1,n-1) \ltimes \R^n$.

\begin{proposition}[\cite{frances_localdynamics}, Prop.4.2]
The conformal diffeomorphism $f$ is linearizable near $x$ if and only if its holonomy is linear (as an affine transformation), i.e. its action on $\R^{1,n-1}$ has a fixed point.
\end{proposition}

Remark that the condition on the holonomy is invariant under conjugacy by elements of $P$, so that the choice of $\hx$ in $\pi^{-1}(x)$ has no influence on it. If $f$ is linearizable, choose a point $\hx$ in the fiber so that its holonomy $p$ is in $\CO(1,n-1)$, \textit{i.e.} has the form
\begin{equation*}
p = 
\begin{pmatrix}
\lambda & & \\
 & M &  \\
 & & \lambda^{-1}
\end{pmatrix}
\end{equation*}
with $\lambda >0$ and $M \in O(1,n-1)$. Because $p \in G_0$, it is not difficult to see that $T_x f$ is conjugate to $\Ad(p)|_{\g_{-1}}$ (see \cite{frances_localdynamics}, proof of Prop. 4.2). At this point, we can already conclude in several cases.

\paragraph*{Strongly stable dynamics.} Consider the Cases \ref{case:(1,2)}.\ref{case2a}, \ref{case:(2,1)}, and \ref{case:(2,3)}. What we have recalled above ensures that $\phi_X^{-t}$ is conjugate near $x_0$ to the flows (for $t \geq 0$)
\begin{equation*}
\begin{pmatrix}
e^{-t} & & \\
&  \ddots & \\
& & e^{-t}
\end{pmatrix}
\text{ (\ref{case:(1,2)}.\ref{case2a}, \ref{case:(2,1)}.) and }
\begin{pmatrix}
e^{-t} & & & & \\
 & e^{-2t} & & & \\
 & & \ddots & & \\
 & & & e^{-2t} & \\
 & & & & e^{-3t}
\end{pmatrix}
\text{ (Case \ref{case:(2,3)}).}
\end{equation*}
Thus, $\phi_X^{-t}$ has \textit{strongly stable} dynamics when $t \rightarrow +\infty$ (see Section 3.2 of \cite{frances_causal_fields}, the notion was first introduced by A. Zeghib in \cite{zeghib99}). By Proposition 4.(iii) of the same paper, we get that a neighborhood of $x_0$ is conformally flat.

\begin{remark}
These dynamics are prototypes of those studied in Frances' paper. In our situation, it is almost direct to verify that the Weyl-Cotton curvature must vanish in a neighborhood of $x_0$.
\end{remark}

\paragraph*{Vanishing of the Weyl-Cotton curvature on the Zero set of $X$.}

We are left to prove conformal flatness in Cases \ref{case:dim=2}. and \ref{case:(1,2)}.\ref{case2b}. In both situations, the flow $\phi_X^{-t}$, $t \geq 0$, is conjugate to 
\begin{equation*}
\begin{pmatrix}
1 & & & & \\
 & e^{-t} & & & \\
 & & \ddots & & \\
 & & & e^{-t} & \\
 & & & & e^{-2t}
\end{pmatrix}
\end{equation*}
This flow is not strongly stable, but just stable and it is not enough to conclude. So, we also consider the behavior of the flow of $Y$ near $x_0$ and use technical properties of conformal flows with non-linear and unipotent holonomy established in \cite{frances_melnick13}.

\vspace*{.2cm}

In Case \ref{case:dim=2}, the holonomy of $Y$ at $x_0$ has the form of a light-like translation of $\Ein^{1,n-1}$.  By Theorem 4.3 of \cite{frances_melnick13}, there exists an open, conformally flat subset $U \subset M$ such that $x_0 \in \bar{U}$. In Case \ref{case:(1,2)}.\ref{case2b}, the holonomy of $Y$ at $x_0$ has the form of the expression (20), Section 5.3 of \cite{frances_melnick13}. By Section 5.3.4 of the same paper, $x_0$ belongs to the closure of some conformally flat open set.

\begin{remark}
In Section 5. of \cite{frances_melnick13}, the authors study conformal vector fields of real-analytic Lorentzian manifolds. However, the real-analytic regularity is not used in the proofs of the two technical facts cited above.
\end{remark}

So, in both cases, the point $x_0$ is in the closure of a conformally flat open subset and by continuity, we get that $W_{x_0} = 0$. So, for the moment, we have come to the

\vspace*{.2cm}

\textbf{Partial conclusion:} If a point admits a $1$-dimensional $H$-orbit, then the Weyl tensor vanishes at this point.

\vspace*{.2cm}

Now, let $U$ be the linearization neighborhood of $\phi_X^t$. The latter admits a segment of fixed points in restriction to $U$. Note it $\Delta$. The holonomy of $Y$ gives us more information thanks to the notion of \textit{development of curves}. Even if it could be explained relatively easily, we will directly use the following property.

\begin{lemma}[Follows from \cite{frances_melnick10}, Prop. 5.3]
\label{lem:fixed_points_holonomy}
Let $x \in M$ and $\hx \in \pi^{-1}(x)$. Let $\gamma(t) = \pi(\exp(\hx,tX_0)) \in M$ and $\gamma_{\X}(t) = \pi_{\X}(e^{tX_0})$, where $\pi_{\X} : G \rightarrow \X = G/P$ is the natural projection, $X_0 \in \g$ and $t \in ]-\epsilon,\epsilon[$ with $\epsilon$ sufficiently small. Let $f \in \Conf(M,g)$ fixing $x$ and having holonomy $p$ at $\hx$. If the left action of $p$ on $\X$ fixes pointwisely the curve $\gamma_{\X}$, then the action of $f$ on $M$ fixes pointwisely $\gamma$.
\end{lemma}

We can now see that in both cases, $Y$ also vanishes on the curve $\Delta$.

\begin{itemize}
\item \textbf{Case \ref{case:dim=2}.} Here, we have a non-zero $X_{-\alpha-2\beta} \in \g_{-\alpha-2\beta}$ such that $[X_h,X_{-\alpha-2\beta}]=0$. Thus, for all $s,t \in \R$, $e^{tX_h} e^{sX_{-\alpha - 2\beta}} = \e^{sX_{-\alpha -2 \beta}} \e^{tX_h}$. This proves that the curve $s \mapsto \pi(\exp(\hx_0,sX_{-\alpha - 2 \beta})$ coincides with $\Delta$ in a neighborhood of $x_0$. Moreover, since $Y_h \in \g_{\alpha}$, we also have $[Y_h,X_{-\alpha-2\beta}] = 0$. So, $\phi_Y^t$ also fixes pointwisely $\Delta$ near $x_0$.

\item \textbf{Case \ref{case:(1,2)}.\ref{case2b}} Here, we have a non-zero $X_{-\alpha} \in \g_{-\alpha}$ such that $[X_h,X_{-\alpha}] = 0$. The same reasoning as above gives that $\Delta$ coincides locally with the curve $s \mapsto \pi( \exp(\hx_0,sX_{-\alpha}))$. 

We have $Y_{\beta}$ and $Y_{\alpha + \beta}$ such that $Y_h = Y_{\alpha+\beta} + Y_{\beta}$ and $[Y_{\alpha+\beta},Y_{\beta}] = 0$. Neither $\alpha - \beta$ nor $2\beta$ are restricted roots. So, $e^{X_{-\alpha}}$ and $e^{Y_{\beta}}$ commute and since $\Ad(e^{tX_{\alpha + \beta}})X_{-\alpha} = X_{-\alpha} +\underbrace{ t[X_{\alpha + \beta},X_{-\alpha}]+(t^2/2)[X_{\alpha + \beta},[X_{\alpha + \beta},X_{-\alpha}]]}_{\in \g_{\beta} \oplus \g_{\alpha + 2\beta}}$ we have
\begin{equation*}
e^{tX_{\alpha + \beta}} e^{sX_{-\alpha}} e^{-tX_{\alpha + \beta}} = e^{sX_{-\alpha}} \underbrace{e^{s(t[X_{\alpha + \beta},X_{-\alpha}]+(t^2/2)[X_{\alpha + \beta},[X_{\alpha + \beta},X_{-\alpha}]])}}_{\in P}.
\end{equation*}
and finally $e^{tY_h}e^{sX_{-\alpha}} = e^{sX_{-\alpha}} p(s,t)$, with $p(s,t) \in P$. According to Lemma \ref{lem:fixed_points_holonomy}, we get that $\phi_Y^t$ fixes each point of the conformal geodesic $\pi(\exp(\hx_0,sX_{-\alpha}))$, that coincides with $\Delta$ in a neighborhood of $x_0$.
\end{itemize}

So, in both cases, the vector fields $X$ and $Y$ vanishes on $\Delta$ near $x_0$. Since $\dim(H.x_0) = 1$, any point in a neighborhood of $x_0$ has an $H$-orbit of dimension at least $1$. So, reducing $U$ if necessary, we have that for all $x \in \Delta$, $\dim H.x = 1$. By the previous partial conclusion, we know that $W$ vanishes in restriction to $\Delta$.

\paragraph*{Conclusion.}

Finally, $\phi_X^{-t}$ has a stable dynamics when $t \rightarrow +\infty$, and for all $x \in U$, $\phi_X^{-t}(x) \rightarrow x_{\infty} \in \Delta$, with $W_{x_{\infty}} = 0$. By Proposition 4. (i) of \cite{frances_causal_fields}, we obtain that $W|_U = 0$, proving that a neighborhood of $x_0$ is conformally flat in Cases \ref{case:dim=2}. and \ref{case:(1,2)}.\ref{case2b}

\subsection{Compact conical orbits}
\label{ss:2_dimensional_orbits}
Let $H$ be a connected Lie group locally isomorphic to $\SL(2,\R)$ that acts conformally on a Lorentzian manifold $(M,g)$. Assume that there exists a point $x_0 \in M$ such that $H.x_0$ is a compact conical orbit, with stabilizer $H_{x_0}$. We know that $\Ad_{\h}(H_{x_0}) \simeq \Z \ltimes U$, where $U$ denotes a unipotent one-parameter subgroup and the factor $\Z$ is generated by a non-trival hyperbolic element normalizing $U$. Let $f \in H_{x_0}$ be in the preimage by $\Ad_{\h}$ of this hyperbolic element. The action of $f$ in restriction to the orbit $H.x_0$ will almost completely prescribe its dynamics near the orbit, as the following proposition shows. 

\begin{proposition}
\label{prop:dynamic_of_f}
The conformal diffeomorphism $f$ is linearizable near $x_0$: there exists an open neighborhood of the origin $\mathcal{U} \subset T_{x_0}M$ and $U \subset M$ an open neighborhood of $x_0$, and a diffeomorphism $\psi : \mathcal{U} \rightarrow U$ such that $\psi$ conjugates $T_{x_0}f$ and $f$. Moreover, replacing $f$ by its inverse if necessary, we have a basis $(e_1,\ldots,e_n)$ of $T_{x_0}M$ in which $g_{x_0}$ reads $2x_1 x_n + x_2^2 + \cdots + x_{n-1}^2$, $X_{x_0} = e_1$ and
\begin{equation*}
T_{x_0}f = 
\begin{pmatrix}
1 & & & & \\
 & \lambda & & & \\
 & & \ddots & & \\
 & & & \lambda & \\
 & & & & \lambda^2
\end{pmatrix}
\begin{pmatrix}
1 & & & & \\
 & & & & \\
 & & R & & \\
 & & & & \\
 & & & & 1 
\end{pmatrix}
\end{equation*}
where $0 < \lambda < 1$ and $R$ is a rotation matrix of $\Span(e_2,\ldots,e_{n-1})$. 
\end{proposition}

The eventual ``compact noise'' commutes with the first matrix and has no influence on the dynamics. The arguments that we developed in \cite{article3} in a similar context are easily adaptable here to the dynamics of $f$, and it will not be a difficult problem to prove conformal flatness of a neighborhood of $H.x_0$. Thus, the important point here is to describe the action of $f$, and for this we make a crucial use of the Cartan geometry associated to $(M,[g])$ to reduce the problem to an algebraic question.

\subsubsection{Algebraic description of the holonomy of $f$}

Let $(M,g)$ be a Lorentzian manifold, and $\pi : \hat{M} \rightarrow M$ and $\omega$ be the Cartan bundle and Cartan connection defined by $[g]$. For all $\hx \in \hat{M}$, we have a linear isomorphism $\varphi_{\hx} : T_xM \rightarrow \g / \p$ defined as follows. If $v \in T_xM$, let $\hat{v} \in T_{\hx} \hat{M}$ such that $\pi_* \hat{v} = v$. Then, $\varphi_{\hx}(v)$ is (well-)defined as the projection of $\omega_{\hx}(\hat{v})$ in $\g/\p$. If $\bar{\Ad}$ denotes the representation of $P$ on $\g / \p$ induced by the adjoint representation, then $\varphi_{\hx.p} = \bar{\Ad}(p^{-1}) \varphi_{\hx}$ (\cite{sharpe}, Ch.5 , Th.3.15). There exists a Lorentzian quadratic form $Q$ on $\g/\p$, such that $\bar{\Ad}(P) < \Conf(\g / \p, Q)$ and such that, by construction of the Cartan geometry associated to $(M,[g])$, the map $\varphi_{\hx}$ sends $g_x$ on a positive multiple of $Q$. 

We saw in Section \ref{sss:cartan_decomposition} that $\g$ admits a grading $\g = \g_{-1} \oplus \g_0 \oplus \g_1$, where $\p = \g_0 \oplus \g_1$, and $\g_{-1} = \g_{-\alpha} \oplus \g_{-\alpha -\beta} \oplus \g_{-\alpha - 2\beta}$. Moreover, $P$ contains a Lie subgroup $G_0$ with Lie algebra $\g_0$ and such that $P \simeq G_0 \ltimes \g_1$ (\cite{cap_slovak}, Prop.1.6.3). Identifying $\g / \p \simeq \g_{-1}$, the lines $\g_{-\alpha}$ and $\g_{-\alpha -2\beta}$ are isotropic with respect to $Q$, and the orthogonal of the Lorentzian plane they span is $\g_{-\alpha -\beta}$. We choose a basis of $(e_1,\ldots,e_n)$ of $\g_{-1}$ such that $e_1 \in \g_{-\alpha}$, $\g_{-\alpha-\beta} = \Span(e_2,\ldots,e_n)$, $e_n \in \g_{-\alpha-2\beta}$, and in which $Q$ reads $2x_1x_n +x_2^2 + \cdots + x_{n-1}^2$. The adjoint action of $G_0$ preserves $\g_{-1}$ and, in the basis we chose, gives an identification $G_0 \simeq \CO(1,n-1) = \R_{>0} \times O(1,n-1)$.

\vspace*{.2cm}

Now, let $H$ be a Lie group locally isomorphic to $\SL(2,\R)$ acting conformally on $(M,g)$, with a compact conical orbit $H.x_0$. Let $f \in H_{x_0}$ be the hyperbolic element we chose at the beginning of this section and let $U=\{e^{tY}, \ t \in \R\} < H_{x_0}$ the unipotent one parameter subgroup normalized by $f$. Diagonalizing $\Ad(f)$, we get $X,Z \in \h$ such that $(X,Y,Z)$ is an $\sl(2)$-triple, with $\Ad(f)X = X$, $\Ad(f) Y = \lambda^{-1} Y$ and $\Ad(f)Z = \lambda Z$, with $\lambda >0$, $\lambda \neq 1$. Since $Y \in \h_{x_0}$, necessarily $X_{x_0}$ is isotropic and orthogonal to $Z_{x_0}$, and $g_{x_0}(Z,Z)>0$ (see Section \ref{sss:tangential_information}).

Let $\hx_0$ be a point in $\pi^{-1}(x_0)$ and let $\iota_{\hx_0}(X) = \omega_{\hx_0}(\hat{X}_{\hx_0})$. Since $X_{x_0}$ is an isotropic vector of $T_{x_0}M$, the projection of $\iota_{\hx_0}(X)$ in $\g/\p$ is isotropic with respect to $Q$. Since $\Ad(G_0)|_{\g_{-1}} \simeq \CO(\g_{-1},Q)$, it acts transitively on the set of isotropic vectors of $\g_{-1}$. Thus, there is $g_0 \in G_0 < P$ such that $\Ad(g_0) \iota_{\hx_0}(X) = \iota_{\hx_0.g_0}(X) \in \g_{-\alpha} + \p$. Hence, there is a choice of $\hx_0$, in the fiber over $x_0$, such that $\iota_{\hx_0}(X) = X_{-\alpha} + X_{\p}$, and we keep this element $\hx_0$. It will be modified in the sequel, but in a way that does not change the projection of $\iota_{\hx_0}(X)$ in $\g / \p$.

\vspace*{.2cm}

Let $p \in P$ be the holonomy of $f$ at $\hx_0$. We have $\Ad_{\g} (p) \iota_{\hx_0}(X) = \iota_{\hx_0.p}(X) = \iota_{\hat{f}(\hx_0)}(X) = \iota_{\hx_0} (\Ad_{\h}(f)(X)) = \iota_{\hx_0}(X)$. So, let us define
\begin{equation*}
P^{\hx_0} = \{p' \in P \ | \ \Ad(p')\iota_{\hx_0}(X) = \iota_{\hx_0}(X)\}.
\end{equation*}
It is an algebraic subgroup of $P$, and $p \in P^{\hx_0}$. Remark that for all $p' \in P$, $P^{\hx_0.p'} = p' P^{\hx_0}p'^{-1}$.

\paragraph*{Stabilizer of $X_{-\alpha}$ modulo $\p$.}

According to the decomposition $P = G_0 \ltimes \g_1$, every element of $P$ can be written $p' = g_0 \exp(Z_1)$, with $g_0 \in G_0$ and $Z_1 \in \g_1$. Now, $[\g_1,\g_{-1}] \subset \g_0$, so $\bar{\Ad}(\exp(\g_1))$ is trivial on $\g/\p$, and $\bar{\Ad}(p') = \bar{\Ad}(g_0) = \Ad(g_0)|_{\g_{-1}}$ if we identify $\g / \p $ and $\g_{-1}$. Thus, $\bar{\Ad}(p')$ fixes $X_{-\alpha} \text{ mod.} \, \p$ if and only if $\Ad(g_0)$ fixes $X_{-\alpha}$. If we reuse the decomposition of Section \ref{sss:cartan_decomposition}, we see that an element $g_0$ fixing $X_{-\alpha}$ has the form
\begin{equation}
\label{eq:g_0}
g_0 = 
\begin{pmatrix}
x & & & & \\
& x & & & \\
& & k & & \\
& & & x^{-1} & \\
& & & & x^{-1}
\end{pmatrix}
\exp(T_{\beta})
\end{equation}
with $x >0$, $k \in M \simeq O(n-2)$ and $T_{\beta} \in \g_{\beta}$.

\paragraph*{Conformal distortion.}

The group $P^{\hx_0}$ being algebraic, we can consider the Jordan decomposition of $p$: it decomposes into a commutative product $p = p_hp_up_e$ of hyperbolic, unipotent and elliptic elements of $P^{\hx_0}$ (\cite{morris}, Section 4.3). Write $p_h = g_0^h \exp(Z_1^h)$, $p_u = g_0^u \exp(Z_1^u)$, $p_e =g_0^e \exp(Z_1^e)$. Since $\bar{\Ad} : P \rightarrow \CO(\g / \p, Q)$ is an algebraic morphism, $g_0^h$, $g_0^u$ and $g_0^e$ are respectively hyperbolic, unipotent and elliptic elements of $G_0$. Thus, we necessarily have
\begin{align*}
& g_0^h = 
\begin{pmatrix}
x_h & & & & \\
& x_h & & & \\
& & k_h & & \\
& & & x_h^{-1} & \\
& & & & x_h^{-1}
\end{pmatrix}
\exp(T_{\beta}^h)
,
\quad
&
g_0^u =
\begin{pmatrix}
1 & & & & \\
& 1 & & & \\
& & k_u & & \\
& & & 1 & \\
& & & & 1
\end{pmatrix}
\exp(T_{\beta}^u)
\\
& g_0^e =
\begin{pmatrix}
1 & & & & \\
& 1 & & & \\
& & k_e & & \\
& & & 1 & \\
& & & & 1
\end{pmatrix}
\exp(T_{\beta}^e)
,
\quad
&
\end{align*}
with $k_h$, $k_u$ and $k_e$ respectively hyperbolic, unipotent and elliptic elements of $O(n-2)$. Thus, we have $k_h=k_u=I_{n-2}$. Moreover, the map $\varphi_{\hx_0}$ conjugates $T_{x_0}f$ to $\bar{\Ad}(p)$. We deduce that $x_h^{-2}$ is the conformal distortion of $T_{x_0}f$. We have $\lambda > 0$, $\lambda \neq 1$ such that $\Ad(f) Z = \lambda Z$, implying $T_{x_0}f . Z_{x_0} = \lambda Z_{x_0}$. Since $g_{x_0}(Z,Z) > 0$, the conformal distortion of $f$ at $x_0$ is equal to $\lambda^2$. This proves that $x_h = \lambda^{-1} \neq 1$. Replacing $f$ by its inverse if necessary, we assume that $\lambda \in ]0,1[$.

\paragraph*{Hyperbolic component.}

If we let $T_{\beta}^0 : = \frac{1}{1-\lambda}T_{\beta}^h$ and $p_{\beta} = \exp(T_{\beta}^0)$, we obtain that 
\begin{equation*}
p_{\beta} p_h {p_{\beta}}^{-1} = 
\begin{pmatrix}
\lambda^{-1} & & & & \\
& \lambda^{-1} & & & \\
& & I_{n-2} & & \\
& & & \lambda & \\
& & & & \lambda
\end{pmatrix}
\exp(\Ad(p_{\beta})Z_1^h).
\end{equation*}
This choice of conjugacy comes in fact from an interpretation of $P$ as the (affine) conformal group of $\R^{1,n-1}$. Now, let $(x_1,\ldots,x_n)$ be the coordinates of $\Ad(p_{\beta})Z_1^h$, seen as a vector of $(\R^n)^*$, \textit{i.e.}
\begin{equation*}
\Ad(p_{\beta})Z_1^h
=
\begin{pmatrix}
0 & x_1 & x_2 & \cdots & x_{n-1} & x_n & 0 \\
  & 0 & 0 & \cdots & 0 & 0 & - x_n \\
  & & & & &0 & -x_2 \\
  & & & & &\vdots & \vdots \\
  & & & & &0 & -x_{n-1} \\
  & & & & &0 & -x_1 \\
  & & & & & & 0
\end{pmatrix}
\end{equation*}
Then, the $2 \times 2$ block in the upper left corner of $p_{\beta}p_h{p_{\beta}}^{-1}$ is
\begin{equation*}
\begin{pmatrix}
\lambda^{-1} & \lambda^{-1} x_1 \\
0 & \lambda^{-1}
\end{pmatrix}
.
\end{equation*}
Since $p_{\beta} p_h {p_{\beta}}^{-1}$ is $\R$-split, this block matrix must be $\R$-split, and we get $x_1 = 0$. So, if we choose $Z_1^1 = (0,\frac{1}{1-\lambda}x_2,\ldots,\frac{1}{1-\lambda}x_{n-1},\frac{1}{1-\lambda^2}x_n)$ and let $p_1 = \exp(Z_1^1)$ then
\begin{equation*}
p_1p_{\beta}p_h{p_{\beta}}^{-1}p_1^{-1} = 
\begin{pmatrix}
\lambda^{-1} & & & & \\
 & \lambda^{-1} & & & \\
 & & I_{n-2} & & \\
 & & & \lambda & \\
 & & & & \lambda
\end{pmatrix},
\end{equation*}
with $p_{\beta} \in \exp(\g_{\beta})$ and $p_1 \in \exp(\g_{\alpha + \beta} \oplus \g_{\alpha + 2\beta}) \subset \exp(\g_1)$. Note that the adjoint actions $\bar{\Ad}(p_{\beta})$ and $\bar{\Ad}(p_1)$ on $\g / \p$ fix the projection of $\g_{-\alpha}$. So, let us replace $\hx_0$ by $\hx_0.(p_1p_{\beta})^{-1}$. The component of $\iota_{\hx_0}(X)$ on $\g_{-1}$ is still $X_{-\alpha}$ and $p_h$ has now the diagonal form we have exhibited above.

\paragraph*{Trivial unipotent component.} 

As we observed before, the decomposition of $p_u$ according to $P = G_0 \ltimes \g_1$ is $p_u = \exp(T_{\beta}^u) \exp(Z_1^u)$. Let us decompose $Z_1^u = Z_{\alpha}^u + Z_{\alpha + \beta}^u + Z_{\alpha + 2\beta}^u$, the indices indicating in which root-spaces the elements are. Using the fact that $p_h \in A = \exp(\a)$, we see that
\begin{align*}
p_hp_up_h^{-1} & = \exp(\Ad(p_h)T_{\beta}^u) \exp(\Ad(p_h)(Z_{\alpha}^u + Z_{\alpha+\beta}^u + Z_{\alpha + 2\beta}^u)) \\
                  & = \exp(\lambda^{-1} T_{\beta}^u). \exp(Z_{\alpha}^u + \lambda^{-1} Z_{\alpha+\beta}^u + \lambda^{-2} Z_{\alpha + 2\beta}^u)
\end{align*}
Since $p_h$ and $p_u$ commute, by uniqueness of the decomposition $P = G_0 \ltimes \g_1$, we get $T_{\beta}^u = 0$, $Z_{\alpha+\beta}^u = 0$ and $Z_{\alpha + 2\beta}^u = 0$. So, $p_u = \exp(Z_{\alpha}^u)$. 

\vspace*{.2cm}

We finally consider $\Ad(p^u)X_{-\alpha}$ modulo $\g_1$. Let us write $\iota_{\hx_0}(X) = X_{-\alpha} + X_0 \text{ mod.} \, \g_1$, where $X_0 \in \g_0$. On the one hand, we have $\Ad(p^u) \iota_{\hx_0}(X) = \iota_{\hx_0}(X)$. Since $p^u \in \exp(\g_1)$, we have $\Ad(p^u) X_0 = X_0 \text{ mod.} \, \g_1$. Thus, $\Ad(p^u) \iota_{\hx_0}(X) = \Ad(p_u)X_{-\alpha} + X_0 \text{ mod.} \, \g_1$, and we obtain
\begin{equation*}
\Ad(p^u)X_{-\alpha} = X_{-\alpha} \text{ mod.} \, \g_1.
\end{equation*}
But on the other hand, since $\g_{\pm \alpha}$ has dimension $1$, $Z_{\alpha}^u$ is a multiple of $\theta X_{-\alpha}$. So, if $Z_{\alpha}^u = \mu \theta X_{-\alpha}$ with $\mu \in \R$, by Proposition 6.52.(a) of \cite{knapp}, we have $[X_{-\alpha},Z_{\alpha}^u] = \mu B_{\theta}(X_{-\alpha},X_{-\alpha}) A_{\alpha}$, where $A_{\alpha} \in \a$ is the element associated to $\alpha$ by the Killing form $B$ and $B_{\theta} = -B(\theta.,.)$. So,
\begin{align*}
\Ad(e^{Z_{\alpha}^u}) X_{-\alpha} & = X_{-\alpha} + \underbrace{[Z_{\alpha}^u,X_{-\alpha}]}_{ \in \a} + \frac{1}{2} \underbrace{[Z_{\alpha}^u,[Z_{\alpha}^u,X_{-\alpha}]]}_{\in \g_{\alpha}} \\
                                  & = X_{-\alpha} - \mu B_{\theta}(X_{-\alpha},X_{-\alpha}) A_{\alpha} \text{ mod.} \, \g_1.
\end{align*}
So, we must have $\mu = 0$, \textit{i.e.} $p^u = \id$.

\paragraph*{Elliptic component.} 

Consider now $P$ as the conformal group of $\R^{1,n-1}$. Since $p_h$ has a diagonal form, its conformal affine action fixes a vector $v_0 \in \R^{1,n-1}$. As for any elliptic element of $\SL(N,\R)$, the set $\{(p_e)^n, \ n \in \Z\}$ is relatively compact in $P$. Therefore, the orbit of $v_0$ under iterations of $p_e$ is also relatively compact. Consider the convex hull
\begin{equation*}
C = \text{Conv}\left ( \overline{ \{(p_e)^n. v_0, \ n\in \Z \} } \right ) \subset \R^{1,n-1}.
\end{equation*}
It is a compact, $p_e$-invariant, convex subset of $\R^{1,n-1}$. Since $p_e$ acts affinely, it has a fixed point in $C$. Moreover, $p_h$ commutes with $p_e$ and fixes $v_0$. So, it fixes every point of $C$. Thus, $p_e$ and $p_h$ admit a common fixed point in $\R^{1,n-1}$: it is then a fixed point for $p = p_hp_e$, proving that $f$ is linearizable near $x_0$.

\paragraph*{Derivative of $f$.} 

As we recalled above the derivative $T_{x_0}f$ is conjugate by $\varphi_{\hx_0}$ to the adjoint action $\bar{\Ad}(p)$ on $\g/\p$, which is the commutative product $\bar{\Ad}(p_u) \bar{\Ad}(p_e)$. In the basis of $\g_{-1}$ we choose at the beginning of this section, we have
\begin{equation*}
\bar{\Ad}(p_h) =
\begin{pmatrix}
1 & & & & \\
 & \lambda & & & \\
 & & \ddots & & \\
 & & & \lambda & \\
 & & & & \lambda^2
\end{pmatrix},
\end{equation*}
the eigenspaces for $1$ and $\lambda^2$ being the projections of $\g_{-\alpha}$ and $\g_{-\alpha-2\alpha}$ respectively. Since $\bar{\Ad}(p_e)$ commutes with $\bar{\Ad}(p_u)$, it preserves the lines $\g_{-\alpha} \text{ mod.} \, \p$ and $\g_{-\alpha - 2\beta} \text{ mod.} \, \p$. The standard form of linear Lorentzian isometries fixing two isotropic lines finally gives the desired form of $\bar{\Ad}(p_e)$.

\subsubsection{Vanishing of the Weyl-Cotton tensor near $x_0$}

Now that the action of $f$ near $x_0$ has been determined, we can prove that $x_0$ is contained in a conformally flat open subset. The arguments are basically the same than those of Section 4.1 of \cite{article3}. We summarize them briefly. The first step is to see that the Weyl curvature vanishes in restriction to the orbit of $x_0$. We note $(e_1,\ldots,e_n)$ the basis given by Proposition \ref{prop:dynamic_of_f} and $\mathcal{H} = \Span(e_2,\ldots,e_n)$. Using the fact that the $(3,1)$-Weyl tensor is $f$-invariant and considering the contraction rates, we see that
\begin{enumerate}
\item $W_{x_0}(\mathcal{H},\mathcal{H},\mathcal{H}) = 0$
\item $W_{x_0}(T_{x_0}M,T_{x_0}M,T_{x_0}M) \subset \mathcal{H}$.
\end{enumerate}
The point is then the following fact.

\begin{lemma}[\cite{article3}, Lemma 4.5]
\label{lem:two_hyperplanes}
Let $\mathcal{H}_1$ and $\mathcal{H}_2$ be two degenerate hyperplanes of $T_{x_0}M$. Assume that they both satisfy points 1. and 2. above. Then, $\mathcal{H}_1 \neq \mathcal{H}_2 \Rightarrow W_{x_0} = 0$.
\end{lemma}

So, if we had $W_{x_0} \neq 0$, then we would have $T_{x_0}\phi_Y^t \mathcal{H} = \mathcal{H}$ because the properties involved in the previous lemma are conformal. Thus, the derivative $T_{x_0}\phi_Y^t$ would preserve $\mathcal{H} \cap T_{x_0} (H.x_0)$, which is a space-like line in $T_{x_0} (H.x_0)$. It is then immediate to see that this is not possible, proving that $W|_{H.x_0} \equiv 0$.

\vspace*{.2cm}

Finally, the fixed points of $f$ near $x_0$ form a segment $\Delta$, that coincides with the orbit $\{\phi_X^t(x_0)\}$. In particular, $W|_{\Delta} \equiv 0$, and we are in a discrete version of the conformal dynamics exhibited in Cases 1. and 2.b. in Section \ref{sss:dynamics_of_X}. Similarly, we can apply Proposition 4.(i) of \cite{frances_causal_fields} to conclude that a neighborhood of $x_0$ is conformally flat.

\subsection{Fixed points}
\label{ss:fixed_points}
Let $(M,g)$ be a compact Lorentzian manifold with a conformal action of $H \simeq_{\text{loc}} \SL(2,\R)$. We assume here that there exists a point $x_0$ fixed by all elements of $H$, and prove that a neighborhood of $x_0$ is conformally flat. To do this, we will use the following property, that is essentially based on the linearizability of conformal actions of simple Lie groups near a fixed point. Its proof uses similar arguments as in \cite{article3}, Section 3. In Corollary 3.4 of the same article, we observed that necessarily $H \simeq \PSL(2,\R) \simeq \SO_0(1,2)$.

\begin{proposition}
\label{prop:no_fixed_points}
Let $x_0$ be a fixed point of the action. There exists an open neighborhood $W$ of $x_0$ and $W' \subset W$ an open-dense subset such that for all $x \in W$, $\dim H.x = 0$ or $2$ and for all $x \in W'$, $\bar{H.x}$ does not contain fixed points.
\end{proposition}

Assume that this proposition is established. Let $x \in W'$. According to Section \ref{s:minimal_subsets}, any minimal $H$-invariant subset $K \subset \bar{H.x} \subset F_{\leq 2}$ is either a compact conical orbit, or a circle. In any event, thanks to Sections \ref{ss:1_dimensional_orbits} and \ref{ss:2_dimensional_orbits}, there exists $x' \in \bar{H.x}$ admitting a conformally flat neighborhood $V$. If $h \in H$ is such that $h.x \in V$, then $h^{-1}V$ is a conformally flat neighborhood of $x$.

This proves that $W'$ is conformally flat, and by continuity of the Weyl-Cotton curvature, all of $W$ is conformally flat. Thus, it is enough to prove Proposition \ref{prop:no_fixed_points} to conclude that a neighborhood of $x_0$ is conformally flat.

\paragraph*{Local orbits near a fixed point.} To do so, we reintroduce the notations of Section 3.3 of \cite{article3}. We fix a basis $(e_1,\ldots,e_n)$ of $T_{x_0}M$ such that $g_{x_0}$ reads $-x_1^2 + x_2^2 + \cdots + x_n^2$ and such that the isotropy representation has the form
\begin{equation*}
A \in \SO_0(1,2) \mapsto	
\begin{pmatrix}
A & \\
 & \id
\end{pmatrix}
\in \SO_0(1,n-1).
\end{equation*}
Let $E$ denote $\Span(e_1,e_2,e_3)$. By the linearizability of conformal actions of simple Lie groups fixing a point, there exists $\mathcal{U} \subset \mathcal{U}' \subset E$ and $\mathcal{V} \subset E^{\perp}$ neighborhoods of the origin, a neighborhood $W$ of $x_0$ in $M$, a neighborhood $V_H \subset H$ of the identity and a diffeomorphism $\psi : \mathcal{U}' \times \mathcal{V} \rightarrow W \subset M$ such that $\psi(0,0) = x_0$, $\forall h \in V_H$, $\rho_{x_0}(h) ( \mathcal{U} \times \mathcal{V}) \subset \mathcal{U}' \times \mathcal{V}$ and
\begin{equation}
\label{equ:linearizability}
\forall (u,v) \in \mathcal{U} \times \mathcal{V}, \psi(\rho_{x_0}(h)(u,v)) = h. \psi(u,v).
\end{equation}
Reducing the open sets if necessary, we assume that $\mathcal{U}$, $\mathcal{U}'$ (resp. $\mathcal{V}$) are open balls in $E$ (resp. $E^{\perp}$) with respect to $x_1^2+x_2^2+x_3^2$ (resp. $x_4^2+\cdots+x_n^2$). Note $q = -x_1^2 + x_2^2 + x_3^2$ the quadratic form induced by $g_{x_0}$ on $E$.

\vspace*{.2cm}

We claim that it is enough to set
\begin{equation*}
W' = \psi((\mathcal{U} \cap \{q \neq 0\}) \times \mathcal{V} ),
\end{equation*}
\textit{i.e.} the union of all local $H$-orbits of type $\H^2$ and $\dS^2$, with the terminology of \cite{article3}, Section 3.3. The point is that Lemma 3.7 of the same paper is in fact valid for any local orbits, not just local $H$-orbits of type $\H^2$. Let us explain how it can be adapted to local orbits of type $\dS^2$. The minor difference is that contrarily to $S_e$, $S_h$ has index $2$ in its normalizer in $\SO_0(1,2)$. If note
\begin{equation*}
h_{\theta} =
\begin{pmatrix}
1 & \\
 & R_{\theta}
\end{pmatrix}
\in \SO_0(1,2),
\end{equation*}
where $R_{\theta}$ denotes the rotation of angle $\theta$ in $\Span(e_2,e_3)$, then the normalizer $N_H(S_h)$ is spanned by $h_{\pi}$ and $S_h$. We reuse the notation
\begin{equation*}
\forall v \in \mathcal{V}, \ \Delta_S(v) = \{\psi(se_3,v) , \ s \in ]0,\epsilon[\},
\end{equation*}
where $\epsilon$ denotes the radius of the ball $\mathcal{U} \subset E$. Every local $H$-orbit of type $\dS^2$ in $W$ meets a unique $\Delta_S(v)$ at a unique point. For all $s$ and $v$, the circle $\{\rho_{x_0}(h_{\theta})(se_3,v), \ \theta \in \R\}$ is included in $\mathcal{U} \times \mathcal{V}$. So, property (\ref{equ:linearizability}) above ensures that for all $\theta$, $h_{\theta} \psi(se_3,v) = \psi(s h_{\theta}e_3,v)$. In particular, $h_{\pi}$ does not fix any point $x \in \Delta_S(v)$, proving that $H_x = S_h$. The proof of Lemma 3.7 of \cite{article3} is now directly adaptable do local orbits of type $\dS^2$.

\vspace*{.2cm}

Let $x \in W'$ and let $x_1$ be a fixed point. Of course, the local description of the action of $H$ that we have made above is valid in the neighborhood of $x_1$. Let $W_1$ denote an analogous neighborhood and assume that $(H.x) \cap W_1 \neq \emptyset$. If $y$ is a point in this intersection, then its stabilizer is conjugate either to $S_e$ or $S_h$. It implies that $y$ belongs to a local orbit of type $\H^2$ or $\dS^2$ in $W_1$. By Lemma 3.7 of \cite{article3}, we get that $(H.y) \cap W_1$ is reduced to the local $H$-orbit of $W_1$ containing $y$. Since this local $H$-orbit is a locally closed submanifold of $W_1$, which does not contain $x_1$, we necessarily have $x_1 \notin \bar{H.x}$. This finishes the proof of Proposition \ref{prop:no_fixed_points}.

\section{Extending conformal flatness everywhere}
\label{s:extending_flatness}
Let $H$ be a Lie group locally isomorphic to $\SL(2,\R)$ acting conformally and essentially on a compact Lorentzian manifold $(M,g)$. We still note $F_{\leq 2}$ the compact, $H$-invariant subset of $M$ where the $H$-orbits have dimension at most $2$. We have seen that any minimal closed $H$-invariant subset of $F_{\leq 2}$ admits a conformally flat neighborhood. It is in fact immediate that all of $F_{\leq 2}$ is contained in a conformally flat open subset: if $x \in F_{\leq 2}$, then $\bar{H.x} \subset F_{\leq 2}$ and contains a minimal $H$-invariant subset $K_x$. If $V$ is a conformally flat neighborhood of $K_x$, there is $h \in H$ such that $h.x \in V$, and $h^{-1}V$ is a conformally flat neighborhood of $x$.

\subsection{Orbits whose closure meets $F_{\leq 2}$}

We are now going to refine this observation. Define
\begin{equation*}
U = \{x \in M \ | \ \bar{H.x} \cap F_{\leq 2} \neq \emptyset \}.
\end{equation*}
\begin{lemma}
$U$ is an open, conformally flat neighborhood of $F_{\leq 2}$.
\end{lemma}

\begin{proof}
By considering a minimal $H$-invariant subset in $\bar{H.x} \cap F_{\leq 2}$, the same argument as above immediately gives that any point in $U$ admits a conformally flat neighborhood. The important point here is that $U$ is open. We denote by $\Int(F_{\leq 2})$ the interior of $F_{\leq 2}$.

\vspace*{.2cm}

Let $x \in U \setminus \Int (F_{\leq 2})$. The closed $H$-invariant subset $\bar{H.x} \cap F_{\leq 2}$ is non-empty. By Proposition \ref{prop:minimal_subsets}, it must contain an orbit $H.x_0$ that is either a compact-conical orbit, a $1$-dimensional orbit or a fixed point of $H$. Since the interior of $F_{\leq 2}$ is $H$-invariant, we have $x_0 \in \partial F_{\leq 2}$. By Proposition \ref{prop:no_fixed_points}, in the neighborhood of any fixed point, every $H$-orbit is either another fixed point or a $2$-dimensional orbit. So, the set of fixed points is included in $\Int(F_{\leq 2})$, proving that the $H$-orbit of the point $x_0$ is either compact-conical or a $1$-dimensional orbit. By Sections \ref{ss:2_dimensional_orbits} and \ref{ss:1_dimensional_orbits}, we know that there is $X \in \h$ hyperbolic such that:
\begin{itemize}
\item Either $\{\phi_X^t(x_0), \ t \in \R \}$ is a non-singular periodic orbit of $X$ and if $t_0 > 0$ is such that $\phi_X^{t_0}(x_0) = x_0$, then $\phi_X^{t_0}$ is linearizable near $x_0$ and conjugate to
\begin{equation*}
\begin{pmatrix}
1 & & & & \\
  & \lambda & & & \\
  & & \ddots & & \\
  & & & \lambda & \\
  & & & & \lambda^2
\end{pmatrix}
\begin{pmatrix}
1 & & & & \\
 & & & & \\
 & & R & & \\
 & & & & \\
 & & & & 1
\end{pmatrix}
\end{equation*}
where $\lambda \in ]0,1[$ and $R$ is a rotation matrix. The fixed points of $\phi_X^{t_0}$ in the linearization neighborhood coincide with the circle $\Delta = \{\phi_X^t(x_0), \ t \in \R \}$. In particular, we have $\Delta \subset F_{\leq 2}$ since it is contained in $H.x_0$ and $\dim H.x_0 = 2$ ;

\item Or $X(x_0) = 0$ and $\phi_X^t$ is linearizable near $x_0$ and is conjugate to one of the following linear flows:
\begin{equation*}
\begin{pmatrix}
1 & & & & \\
  & e^{-t} & & & \\
  & & \ddots & & \\
  & & & e^{-t} & \\
  & & & & e^{-2t}
\end{pmatrix}
, \
\begin{pmatrix}
e^{-t} & & \\
 & \ddots & \\
 & & e^{-t}
\end{pmatrix}
, \
\begin{pmatrix}
e^{-t} & & & & \\
  & e^{-2t} & & & \\
  & & \ddots & & \\
  & & & e^{-2t} & \\
  & & & & e^{-3t}
\end{pmatrix}
.
\end{equation*}
\end{itemize}

In the first situation, if $y$ is a point in the linearization neighborhood of $x_0$, then $(\phi_X^{nt_0}(y))\xrightarrow[n \to \infty]{} y_{\infty} \in \Delta \subset F_{\leq 2}$, proving that this neighborhood of $x_0$ is included in $U$.

In the second situation, when $t \to + \infty$, either $\phi_X^t(y) \rightarrow x_0$ for any $y$ in the linearization neighborhood, or $\phi_X^t(y) \rightarrow y_{\infty} \in \Delta'$, where $\Delta'$ denotes the zero-set of $X$. Of course, $\Delta' \subset F_{\leq 2}$, proving $y \in U$.

\vspace*{.2cm}

Thus, in any case, the point $x_0$ is in the interior of $U$. Since $U$ is $H$-invariant, we also have $x \in \Int(U)$. Finally, $U \setminus \Int(F_{\leq 2}) \subset \Int(U)$, proving that $U = \Int(U)$.
\end{proof}

Our aim is to prove that $U = M$. So, until the end of this section, we assume that $U \neq M$, and consider $K := \partial U$. Since $U$ is $H$-invariant, $K$ is non-empty, compact and $H$-invariant. We are going to prove that the dynamics of $H$ must be stable near $K$. This will be in contradiction with the fact that there are points in $U$ arbitrarily close to $K$.

\subsection{Stability of $H$-orbits in a neighborhood of $K$}

Since $U$ is open and $F_{\leq 2} \subset U$, we have $K \cap F_{\leq 2} = \emptyset$, \textit{i.e.} $H$-acts locally freely in a neighborhood of $K$. This observation implies that for any hyperbolic $X \in \h$, the corresponding conformal vector field is space-like in a neighborhood of $K$, as the following lemma shows.

\begin{lemma}
\label{lem:Xspace_like}
Let $(M,g)$ be a Lorentzian manifold on which $H$ acts conformally. Let $K \subset M$ be a compact subset such that $H$ acts locally freely on $K$, i.e. $\h_x = 0$ for all $x \in K$. Assume that there is an hyperbolic element $X \in \h \simeq \sl(2,\R)$ whose flow preserves $K$. Then, $X$ is space-like in a neighborhood of $K$.
\end{lemma}

\begin{proof}
Let $A^+:=\{\e^{tX}\}_{t \in \R} < H$ and consider the compact $A^+$-invariant subset 
\begin{equation*}
K \cap \{x \in M \ | \ g_x(X,X) \leq 0\}.
\end{equation*}
Assume that this subset is non-empty. By Proposition \ref{prop:tangential_information}, it must contain a point $x_0$ such that $\Ad_{\h}(A^+) \subset \Conf(\h,q_{x_0})$.  We note $a_t := \Ad_{\h}(e^{tX})$. Since $a_t$ is linear and conformal with respect to $q_{x_0}$, there exists $\lambda \in \R$ such that $a_t^* q_{x_0} = \e^{\lambda t} q_{x_0}$. Since $X$ is hyperbolic, there exists $Y$ and $Z$ such that $a_t(Y) = e^t Y$ and $a_t(Z) = e^{-t}Z$. We now use the following observation, which was proved in \cite{article3}, Lemma 2.3.
\begin{fact*}
Let $q$ be an $\Ad(e^{tX})$-conformally invariant sub-Lorentzian quadratic form on $\h$. Then, $q$ is Lorentzian, $X$ is space-like and orthogonal to $Y$ and $Z$, which are both light-like.
\end{fact*}
Thus, we get that $g_{x_0}(X,X) > 0$, contradicting $x_0 \in K \cap \{x \in M \ | \ g_x(X,X) \leq 0\}$. Hence, $X$ is space-like on $K$, and necessarily this is true in a neighborhood of $K$.
\end{proof}

Let us fix $(X,Y,Z)$ an $\sl(2)$-triple in $\h$. By Lemma \ref{lem:Xspace_like}, we know that $X$ must be space-like in a neighborhood of $K$. If we note $V = \{x \in M \ | \ g_x(X,X) > 0\}$, let $g_0$ denote the metric $g/g(X,X)$ on $V$. By compactness of $K \subset V$, the functions $g_0(Y,Y)$, $g_0(Z,Z)$, $g_0(Y,X)$ and $g_0(Z,X)$ are bounded over $K$. Therefore, for any $x \in K$, $Y_x$ and $Z_x$ are isotropic and orthogonal to $X_x$ (see the proof of Lemma \ref{lem:couple_inessential}). So, for all $x \in K$, the subspace $\Span(X_x,Y_x,Z_x)$ is Lorentzian.
\vspace*{.2cm}

Hence, $H$ acts locally freely with Lorentzian orbits in a neighborhood of $K$. So, let us define the open set
\begin{equation*}
\Omega = \{x \in M \ | \ \dim(H.x) = 3, \ H.x \text{ Lorentzian}, \ X_x \text{ space-like}\}.
\end{equation*}
We have proved that $K \subset \Omega$. Remark that $\Omega$ is \textit{a priori} only $\phi_X^t$-invariant.

\vspace*{.2cm}

Let us consider the Lorentzian manifold $(\Omega,g)$. This manifold is endowed with an $\sl(2)$-triple $(X,Y,Z)$ of conformal vector fields, everywhere linearly independent, with $\Span(X,Y,Z)$ Lorentzian and such that $X$ is space-like and complete. To simplify notations, we assume that $g$ has been renormalized by $g(X,X) >0$, so that $\phi_X^t \in \Isom(\Omega,g)$ by Lemma  \ref{lem:centralizer_inessential}. Define $\N$ to be the distribution in $\Omega$ orthogonal to $\Span(X,Y,Z)$. It has codimension $3$, is $\phi_X^t$-invariant, and for all $x \in \Omega$, $\N_x$ is a Riemannian subspace of $T_x\Omega$. 

\vspace*{.2cm}

Finally, we define for small enough $\epsilon >0$
\begin{equation*}
K_{\epsilon} = \{\exp_x(v), \ x \in K, \ v \in \N_x, \ g_x(v,v) \leq \epsilon \}.
\end{equation*}
(The notation $\exp$ refers to the exponential map of the metric $g$.)

\begin{lemma}
If $\epsilon$ is small enough, $K_{\epsilon}$ is a (well-defined) $\phi_X^t$-invariant neighborhood of $K$, and for any neighborhood $V$ of $K$, there is $\epsilon >0$ such that $K_{\epsilon} \subset V$.
\end{lemma}

\begin{proof}
Let $h$ be some auxiliary Riemannian metric on $\Omega$. We note $T^1 \Omega$ the unit tangent bundle with respect to $h$. By compactness of $K$, there exists $\alpha>0$ such that
\begin{equation*}
\forall x \in K, \ \forall v \in \N_x, \ g_x(v,v) \geq \alpha h_x(v,v).
\end{equation*}
On can take $\alpha$ to be the infimum of $g_x(v,v)$ over the compact subset $(T^1 \Omega \cap \N)|_K$ of $T\Omega$.

By compactness of $K$, there is $\eta_0 > 0$ such that if $x \in K$ and $v \in T_x \Omega$ is such that $h_x(v,v) \leq \eta_0$, then $v$ is in the injectivity domain of $\exp_x$. Thus, $K_{\epsilon}$ is well-defined as soon as $\epsilon \leq \alpha \eta_0$. If $\eta \leq \eta_0$, let $V_{\eta}$ denote the exponential neighborhood $V_{\eta} = \{ \exp_x(v), \ x \in K, \ v \in T_x \Omega, \ h_x(v,v) \leq \eta\}$. By continuity of the exponential map of $g$ and compactness of $K$, for any neighborhood $V$ of $K$, there is $\eta$ such that $V_{\eta} \subset V$, implying that $K_{\alpha \eta} \subset V$.

\vspace*{.2cm}

We are left to prove that $K_{\epsilon}$ is a neighborhood of $K$. Let $x \in K$. We know that $H.x$ is an immersed $3$-dimensional Lorentzian submanifold of $(M,g)$, and that $H.x \subset K$. Choose $U \subset M$ an open neighborhood of $x$, $\psi : U \rightarrow U_0 \subset \R^n$ a local chart at $x$, and $V \subset H$ a neighborhood of the identity such that $\psi$ maps diffeomorphically $V.x$ onto an open ball $B_0 \subset E_0$ where $E_0$ is a $3$-dimensional subspace of $\R^n$. We note $g_0$ the push-forward by $\psi$ of the metric $g$ on $U_0$. Immediately, $B_0$ is a Lorentzian submanifold of $U_0$ and we note $\N^0$ the push-forward by $\psi$ of the Riemannian distribution $\N$.

Note $x_0 = \psi(x)$. If $\mathcal{V} \subset E_0^{\perp}$ is a small enough neighborhood of the origin, consider the differentiable map
\begin{align*}
\varphi : & \: B_0 \times \mathcal{V}  \rightarrow U_0 \\
          & \quad (y_0,v) \mapsto \exp_{y_0}(v),
\end{align*}
where the notation $\exp$ refers to exponential map of the metric $g_0$. Remark that for any $y_0 \in B_0$, we have $\mathcal{N}_{y_0}^0 = E_0^{\perp}$. It is then immediate that $T_{(x_0,0)} \varphi$ is inversible, so that $\varphi$ is a local diffeomorphism at $(x_0,0)$. So, there is an open neighborhood $U_0'$ of $x_0$ that is contained in the image of $\varphi$. By construction, this means that $\psi^{-1}(U_0') \subset K_{\epsilon}$, proving that $K_{\epsilon}$ is a neighborhood of $x$, for any $x \in K$.
\end{proof}

The last ingredient leading to a contradiction is the following fact.

\begin{lemma}
The action of $H$ preserves $g$ in a neighborhood of $K$.
\end{lemma}

\begin{proof}
If $\epsilon$ is small enough, $K_{\epsilon}$ is relatively compact in $\Omega$. Since $K_{\epsilon}$ is $\phi_X^t$-invariant, the  functions $g(X,Y)$, $g(X,Z)$, $g(Z,Z)$ and $g(Y,Y)$ are bounded along the orbits of $\phi_X^t$ in $K_{\epsilon}$. Thus, we can apply Lemma \ref{lem:couple_inessential} to the couples of conformal vector fields $(X,Y)$ and $(X,Z)$ and conclude that $X,Y,Z$ are Killing vector fields of $(\Int(K_{\epsilon}),g)$, where $\Int(K_{\epsilon})$ denotes the interior of $K_{\epsilon}$.
\end{proof}

We can now finish the proof. If $\epsilon >0$ is chosen small enough, $K_{\epsilon}$ is included in the neighborhood of $K$ on which $H$ acts by isometries of $g$. Therefore, all of $H$ preserves these $K_{\epsilon}$'s. On the one hand, we always have $K_{\epsilon} \cap U \neq \emptyset$ since $K = \partial U$. So, by $H$-invariance of $K_{\epsilon}$, we obtain that $\bar{K_{\epsilon}} \cap F_{\leq 2} \neq \emptyset$, by definition of $U$.

But on the other hand, since $K \cap F_{\leq 2} = \emptyset$, these compact subsets can be separated by open neighborhoods. So, there exists a neighborhood $V$ of $F_{\leq 2}$ such that for small enough $\epsilon > 0$, $K_{\epsilon} \cap V = \emptyset$. This is our contradiction.

\subsection{Conclusion}

Finally, $U=M$, \textit{i.e.} for all $x \in M$, $\bar{H.x} \cap F_{\leq 2} \neq \emptyset$. Dynamically, this proves that there does not exist a compact $H$-invariant subset of $M$ in which all orbits are $3$-dimensional, and completes the proof of Proposition \ref{prop:minimal_subsets}.

\vspace*{.2cm}

At a geometrical level, since we already know that $F_{\leq 2}$ is contained in a conformally flat open subset, this proves that $(M,g)$ is conformally flat, and completes the proof of Theorem \ref{thm:main}.

\section*{Appendix}
We give here a justification to the following lemma, used for Corollary \ref{cor:su1k}.

\begin{lemma}
Let $k \geq 2$ and $n \geq 3$. If $f : \su(1,k) \rightarrow \so(2,n)$ is a Lie algebra embedding, then the centralizer in $O(2,n)$ of the image of $f$ is compact.
\end{lemma}

\begin{proof}
Let $\theta_{\h}$ be a Cartan involution of $\h := \su(1,k)$, and fix $\h = \a_{\h} \oplus \u(k-1) \oplus \h_{\pm \lambda} \oplus \h_{\pm 2\lambda}$ a corresponding restricted root-space decomposition. We have $\dim \h_{\pm \lambda} = 2k-2$, $\dim \h_{\pm 2\lambda} = 1$ and the bracket $\h_{\lambda} \times \h_{\lambda} \rightarrow \h_{2\lambda}$ is such that $\h_{\lambda} \oplus \h_{2\lambda}$ is isomorphic to the Heisenberg Lie algebra of dimension $2k-1$.

Choose $A \in \a_{\h}$. There exists a Cartan involution $\theta_{\g}$ of $\g=\so(2,n)$ such that $f \circ \theta_{\h} = \theta_{\g} \circ f$. In particular, $\a_{\h}$ is sent into a Cartan subspace of $\g$ with respect to $\theta_{\g}$, and up to conjugacy in $O(2,n)$, we get that $\a_{\h}$ is sent into the Cartan subspace $\a_{\g}$ of $\g$ described in Section \ref{sss:cartan_decomposition}, corresponding to the standard Cartan involution of matrices Lie algebras. We reuse the notations of this section.

Write $f(A) = (a,b)$. Then, using the fact that $[f(\h_{\lambda}),f(\h_{\lambda}] = f(\h_{2\lambda})$, with $\dim f(\h_{2\lambda}) =1$, we obtain, by considering exhaustively all the possibilities, that necessarily $(a,b)$ is proportional to $(1,1)$ and that $f(\h_{\lambda}) \subset \g_{\beta} \oplus \g_{\alpha + \beta}$ and $f(\h_{2\lambda}) = \g_{\alpha + 2\beta}$ (of course, up to exchanging $\lambda$ and $-\lambda$). 

Now, let $g \in O(2,n)$ centralizing $f(\h)$. Firstly, since $g$ centralizes $f(A)$, whose form is known, it has the form
\begin{equation*}
g =
\begin{pmatrix}
g_0 & & \\
 & G_0 & \\
 & & g_1
\end{pmatrix}
,
\end{equation*}
with $g_0 \in \GL(2,\R)$, $G_0 \in O(n-2)$ and $g_1= V ~ ^{t} (g_1)^{-1} V$, where 
$V = 
\begin{pmatrix}
0 & 1 \\
1 & 0
\end{pmatrix}
$.

Secondly, using $\Ad(g) f(\h_{2\lambda}) = f(\h_{2\lambda})$, we get $g_0 \in \SL(2,\R)$. To finish, we claim that $g_0$ is in fact elliptic, what will be enough. To observe this, take a non-zero element $X \in \h_{\lambda}$. The matrix block-form of $f(X)$ is
\begin{equation*}
f(X) =
\begin{pmatrix}
0 & U & 0 \\
  & & - ~ ^{t} U \\
  & & 0
\end{pmatrix}
,
\text{ with }
U =
\begin{pmatrix}
u_1 & \cdots & u_{n-2} \\
v_1 & \cdots & v_{n-2}
\end{pmatrix}
.
\end{equation*}
Since we have $[f(X),\theta_{\g}f(X)] = f([X,\theta_{\h}X]) \in f(\a_{\h})$ (\cite{knapp}, Prop. 6.52(a)), and since $f(A)$ is proportional to the diagonal matrix
\begin{equation*}
\begin{pmatrix}
1 & & & & \\
 & 1 & & & \\
 & & 0 & & \\
 & & & -1 & \\
 & & & & -1
\end{pmatrix}
,
\end{equation*}
we obtain that the vectors $u= (u_1,\ldots,u_{n-2})$ and $v = (v_1,\ldots,v_{n-2})$ satisfy $|u| = |v|$ and are orthogonal with respect to the standard Euclidean structure of $\R^{n-2}$. In particular, they are linearly independent.

Finally, the fact $\Ad(g) f(X) = f(X)$ gives $g_0 U = U G_0$, meaning
\begin{equation*}
\begin{cases}
uG_0 & = a u + bv \\
vG_0 & = cu + dv
\end{cases}
,
\text{ where }
g_0 =
\begin{pmatrix}
a & b \\
c & d
\end{pmatrix}
.
\end{equation*}
Thus, $G_0$ preserves the plane spanned by $u$ and $v$ and induces there the linear endomorphism $g_0$. Since $G_0$ is orthogonal, we get that $g_0$ is indeed elliptic.
\end{proof}

\vspace*{1cm}

\bibliographystyle{amsalpha}
\bibliography{bibli_sl2_cas_general.bib}
\nocite{*}

\end{document}